\documentclass[a4paper, reqno]{amsart}

\usepackage{graphicx, psfrag, csquotes}
\usepackage{amsfonts, amsmath, bbm, amssymb, amsthm, mathtools, scalerel}
\usepackage[shortlabels]{enumitem}
\usepackage{hyperref}
\usepackage{cite}
\usepackage{caption, subcaption}
\usepackage{float}
\usepackage[dvipsnames]{xcolor}
\usepackage{sidecap}
\usepackage{mathtools}
\mathtoolsset{showonlyrefs}

\title{Weak coupling for Schrödinger operators with complex potentials}

\author{Jussi Behrndt}
\author{Markus Holzmann}
\author{Petr Siegl}
\author{Nicolas Weber}

\address{Institute of Applied Mathematics, Graz University of Technology, Steyrergasse 30, 8010 Graz, Austria}
\email{behrndt@tugraz.at, holzmann@tugraz.at, siegl@tugraz.at, nicolas.weber@tugraz.at}

\setlist[enumerate]{topsep=3pt, itemsep=3pt, leftmargin=*}
\setlist[itemize]{topsep=3pt, itemsep=3pt}
\setlist[enumerate,1]{label={\upshape(\roman*)}}

\newcommand{\Hbeta}{H_\beta}

\newcommand{\KernelM}[1]{\calM^{(#1)}}
\newcommand{\KernelL}[1]{\calL^{(#1)}}
\newcommand{\KernelQ}[1]{\calQ^{(#1)}}
\newcommand{\ResolventKernel}[1]{\vphantom{R^2}\smash{\calG^{(#1)}}}

\newcommand{\errorR}[1]{\calJ^{(#1)}} 
\newcommand{\funcg}[1]{g_1}
\newcommand{\funch}[1]{g_2}
\newcommand{\funcU}{\calU}
\newcommand{\fbeta}[1]{\Lambda_{\beta}}
\newcommand{\feps}[1]{\Lambda_{\epsilon}}

\newcommand{\Rollnik}[2]{
	\def\arg{#1}\def\one{1}
	\ifx\arg\one
		\calR_{#2}(\R)
	\else
		\calR_{#2}(\R^{#1})
	\fi
}

\newcommand{\dmath}{\mathop{}\!\mathrm{d}}
\newcommand{\dx}{\dmath x}
\newcommand{\dy}{\dmath y}

\newcommand{\dt}{\dmath t}

\newcommand{\R}{\mathbb{R}}
\newcommand{\C}{\mathbb{C}}

\renewcommand{\Im}{\operatorname{Im}}
\renewcommand{\Re}{\operatorname{Re}}

\newcommand{\dom}{\operatorname{dom}}
\newcommand{\kernel}{\operatorname{ker}}
\newcommand{\ran}{\operatorname{ran}}
\renewcommand{\dim}{\operatorname{dim}}
\newcommand{\codim}{\operatorname{codim}}

\newcommand{\argorempty}[1]{\def\arg{#1}\ifx\arg\empty \,\cdot\, \else #1 \fi}
\def\tmpsmall{small}
\def\tmpbig{big}
\def\tmpBig{Big}

\newcommand{\spf}[3][]{\def\arg{#1}
	\ifx
		\arg\tmpsmall ( \argorempty{#2}, \argorempty{#3} )
	\else
		\ifx
			\arg\tmpbig \big( \argorempty{#2}, \argorempty{#3} \big)
		\else
			\ifx
				\arg\tmpBig \Big( \argorempty{#2}, \argorempty{#3} \Big)
			\else
				\left( \argorempty{#2}, \argorempty{#3} \right) 
			\fi
		\fi
	\fi
}

\newcommand{\norm}[2][]{\def\arg{#1}
	\ifx\arg\tmpsmall 
		\lVert #2 \rVert
	\else
		\ifx\arg\tmpbig 
			\big\lVert #2 \big\rVert 
		\else 
			\ifx\arg\tmpBig 
				\Big\lVert #2 \Big\rVert
			\else
				\left\lVert \argorempty{#2} \right\rVert
			\fi
		\fi
	\fi
}

\newcommand{\abs}[2][]{\def\arg{#1}
	\ifx\arg\tmpsmall\vert{#2}\vert
	\else
		\ifx\arg\tmpbig 
			\big\vert{#2}\big\vert
		\else 
			\ifx\arg\tmpBig 
				\Big\vert{#2}\Big\vert
			\else
				\left\lvert\argorempty{#2}\right\rvert
			\fi
		\fi
	\fi
}

\newcommand{\disc}[2]{D(#1;#2)}
\newcommand{\discBig}[2]{D\Big( #1; \: #2 \Big)}

\newcommand{\logabs}[1]{\abs[big]{\log\abs{#1}}}
\newcommand{\spacer}[2]{\phantom{#2}\text{#1}\phantom{#2}}

\newcommand{\calG}{\mathcal{G}}
\newcommand{\calH}{\mathcal{H}}
\newcommand{\calI}{\mathcal{I}}
\newcommand{\calJ}{\mathcal{J}}
\newcommand{\calL}{\mathcal{L}}
\newcommand{\calM}{\mathcal{M}}
\newcommand{\calN}{\mathcal{N}}
\newcommand{\calO}{\mathcal{O}}
\newcommand{\calQ}{\mathcal{Q}}
\newcommand{\calR}{\mathcal{R}}
\newcommand{\calU}{\mathcal{U}}

\let\temp\phi
\let\phi\varphi
\let\varphi\temp
\let\temp\epsilon
\let\epsilon\varepsilon
\let\varepsilon\temp

\let\originalleft\left
\let\originalright\right
\renewcommand{\left}{\mathopen{}\mathclose\bgroup\originalleft}
\renewcommand{\right}{\aftergroup\egroup\originalright}

\counterwithin{equation}{section}

\theoremstyle{plain}
\newtheorem{theorem}{Theorem}[section]
\newtheorem{proposition}[theorem]{Proposition}
\newtheorem{lemma}[theorem]{Lemma}
\newtheorem{corollary}[theorem]{Corollary}

\theoremstyle{definition}

\newtheorem{remark}[theorem]{Remark}
\newtheorem{example}[theorem]{Example}

\newcommand{\const}{R}
\newcommand{\constcor}{R'}
\begin{document}

\subjclass[2020]{Primary: 35J10, 47A10, 47A75, 47B28, 47F05; Secondary: 34L05, 47A55, 47E05}

\keywords{Bound state, weak coupling, asymptotic expansion, Birman-Schwinger principle, complex-valued potential, non-self-adjoint Schrödinger operator}

\date{\today}

\begin{abstract}
We study the discrete eigenvalues emerging from the threshold of
the essential spectrum of one or two-dimensional Schrödinger operators with complex-valued
$ L^p $-potentials in a weak coupling regime.
We derive necessary and sufficient conditions on the potential for the existence or absence
of discrete eigenvalues in this regime and also analyze their uniqueness and algebraic multiplicity.
Our results can be viewed as natural non-self-adjoint extensions of the well-known classical weak coupling phenomenon for self-adjoint Schrödinger operators with real-valued potentials going back half a century to Simon's famous paper \cite{Simon1976}.
\end{abstract}

\maketitle

\section{Introduction and main results}

The analysis of Schrödinger operators
\begin{align*}
H=    -\Delta + V
   \spacer{in}{\quad}
   L^2(\R^d),
\end{align*}
with a complex-valued potential $ V : \R^d \rightarrow \C $ is challenging due to a fundamental lack of variational methods, a spectral theorem, and
related self-adjoint tools. The properties of $H$ are not fully understood in spite of a high research activity in the last two decades, see, e.g., the monographs~\cite{Davies-2007,Helffer-2013,Sjoestrand-2019-14} on non-self-adjoint (differential) operators and a collection of recent papers  \cite{AbramovAslanyianDavies2001,Boegli-2016-352, Boegli-2021-11, Boegli-2022-398, Cuenin-2019-376, Cuenin-2022-392, Davies-2002-148, Demuth-2009-257, Demuth-2013-232, Enblom-2015-106, Faupin-2023-6, Fanelli-2018-8, Frank-2006-77, Frank2011, Frank2017, Frank2018, Hansmann-2011-98, HansmannKrejcirik2022,Laptev-2009-292,Someyama-2019-83,Stepin-2014-89,Wang-2011-96} on spectral properties
of Schrödinger operators with $L^p$-potentials; in particular,  enclosures of discrete eigenvalues outside the essential spectrum, their absence, number, and accumulation as well as Lieb-Thirring inequalities are investigated therein.

It is well-known (see, e.g., \cite[Lem.~4.2]{Frank2018}) that $V \in L^p(\R^d)$ with
\begin{equation}\label{p.rel.com}
p \in  \begin{cases}
   [1, \infty) & \text{ if } d=1,
   \\
   (1, \infty) & \text{ if } d=2,
   \\
   [d/2, \infty) & \text{ if } d\geq 3,
\end{cases}
\end{equation}
is a relatively form compact perturbation of~$-\Delta$. Hence, for such $V$, the operator $H$ with the domain
$
\dom H= \lbrace f \in H^1(\R^d) : (-\Delta + V) f \in L^2(\R^d) \rbrace
$
is $m$-sectorial, the essential spectrum of $-\Delta$ remains stable under the perturbation by $V$, i.e.,
\begin{equation}
   \sigma_{\rm ess}(H) =     \sigma_{\rm ess}(-\Delta) = [0,\infty),
\end{equation}
and the spectrum of $H$ outside of $[0,\infty)$ comprises exclusively isolated eigenvalues of finite algebraic multiplicities.

More quantitatively, it is known that every
$\lambda \in \sigma_{\rm p}(H) \setminus [0,\infty)$ satisfies
\begin{align}
   |\lambda|^{\frac12} & \leq \frac{1}{2} \int_{\R} |V(x)| \dx & \text{ if } d &= 1, \label{ineq:Davies}
   \\
   |\lambda|^{\kappa} &\leq D_{\kappa,d} \int_{\R^d} |V(x)|^{\kappa+\frac{d}{2}} \dx & \text{ if } d &\geq 2, \quad 0 < \kappa \leq \frac{1}{2}, \label{ineq:Frank}
\end{align}
where $ D_{\kappa,d} > 0 $ are constants independent of $ V $, see \cite{AbramovAslanyianDavies2001, Frank2011}.
Moreover, if 
$ d \geq 3 $ there exist constants $D_{0,d}>0$ for which the condition
\begin{equation}\label{EV.abs.cond}
D_{0,d} \int_{\R^d} |V(x)|^{\frac{d}{2}} \dx < 1
\end{equation}
yields the absence of all eigenvalues, so $\sigma(H) = \sigma_{\rm c}(H)=[0,\infty)$, cf.~\cite[Thm.~2]{Frank2011}, \cite[Thm.~3.2]{Frank2017}, and $H$ is even similar to $-\Delta$, cf.~\cite[Thm.~12]{HansmannKrejcirik2022} as well as the pioneering works \cite[Thm.~XX.2.22]{DS3} and \cite[Thm.~6.1, 6.4]{Kato-1966-162}. For further results on the absence of eigenvalues in any dimension see \cite{Cossetti-2019-379,Fanelli-2018-275,Fanelli-2018-8,KochTataru2006, IonescuJerison2003}.

The main objective of this paper
is to investigate the existence, uniqueness, and multiplicity of eigenvalues of non-self-adjoint Schrödinger operators with
complex-valued potentials in the \emph{weak coupling} regime, i.e.,~the eigenvalues of
\begin{equation}\label{H.beta.def}
\begin{aligned}
H_\beta & := -\Delta - \beta V,
\\
\dom \Hbeta & := \lbrace f \in H^1(\R^d) : (-\Delta - \beta V)f \in L^2(\R^d) \rbrace,
\end{aligned}
\end{equation}
with a \emph{complex}-valued potential $V$
and a \emph{complex} coupling~$\beta \to 0$. This analysis complements and extends well-known classical results on
the weak coupling phenomenon
for self-adjoint Schrödinger operators with
real-valued potentials and a real coupling parameter going back to \cite{Simon1976}; 
we comment on this special case in more detail at the end of the Introduction.
It should be emphasized that weak coupling is a purely one and two-dimensional phenomenon, 
cf.~Remark~\ref{rem3}, and that weakly coupled eigenvalues emerge from the threshold of the essential spectrum.
Moreover, we stress that variational methods are not available for a non-real-valued potential $V$, and hence the existence of eigenvalues of
$H$ is a delicate task in general, see \cite{Pavlov-1962-146,Boegli-2016-352,Cuenin-2022-392,Wang-2011-96} (for the strong coupling regime, see, e.g.,~\cite{Almog-2016-48,Herau-2025-77,Semoradova-2022-54}).
Further results on non-self-adjoint weak coupling can be found in \cite{Borisov-2008-62,Novak-2016-96} in the case of waveguides, in \cite{CueninSiegl2018} for one-dimensional Dirac operators, the related one-dimensional damped wave operators,
and two-dimensional armchair graphene waveguides and in \cite{Fialova-2025-37} for Pauli operator with an Aharonov-Bohm potential. 

In the following we state our main result Theorem~\ref{theorem:s=2} and discuss some interesting special cases in
Corollary~\ref{cor:s=2.V.real} and Corollary~\ref{cor:s=2.V.complex.beta.real} below.
From now on we assume that the complex-valued potential $V$ satisfies (cf.~\eqref{p.rel.com})
\begin{equation} \label{assumption:family_v}
   \begin{aligned}
       V \in
       \begin{cases}
           L^1(\R) & \text{ if } d = 1, \\
           L^1(\R^2) \cap L^{1+\eta}(\R^2) \text{ for some } \eta > 0 & \text{ if } d = 2,
       \end{cases}
   \end{aligned}
\end{equation}
and, in addition, a Rollnik-type integrability condition 
\begin{equation} \label{def:Rollnik_space_mainjussi}
       \int_{\R^{2d}} \abs[]{V(x)} \big| \mu^{(d)}(x-y) \big|^2 \abs[]{V(y)} \dmath(x,y)<\infty,
\end{equation}
where
\begin{align} \label{def:mu_main}
   \mu^{(d)}(x)
   :=
   \begin{cases}
       \frac{1}{2}|x| & \text{ if } d = 1, \\
       \frac{1}{2\pi}\log|x| & \text{ if } d = 2.
   \end{cases}
\end{align}

We also make use of the complex numbers
\begin{align} \label{def:U_U_1_intro}
   U := \int_{\R^d} V(x) \dx\quad\text{and}
   \quad
   U_1 := \int_{\R^{2d}} V(x) \mu^{(d)}(x-y) V(y) \dmath(x,y),
\end{align}
which are well-defined for a potential $ V $ satisfying \eqref{assumption:family_v} and \eqref{def:Rollnik_space_mainjussi};
cf.~\eqref{U_1_exists}.

\begin{theorem} \label{theorem:s=2}
Let $H_\beta = -\Delta -\beta V $ be as in \eqref{H.beta.def} with $V$ satisfying \eqref{assumption:family_v} and
\eqref{def:Rollnik_space_mainjussi}.
For $ U $ and $ U_1 $ in \eqref{def:U_U_1_intro} assume $ U \neq 0 $ and define
\begin{equation} \label{def:funcU_intro}
   \funcU(\beta)
   :=
   \begin{cases}
       U \beta - U_1 \beta^2 & \text{ if } d = 1, \\
       U \beta - \left[ U_1 - (2\pi)^{-1}( \log(2) - \gamma) U^2 \right] \beta^2 & \text{ if } d = 2,
   \end{cases}
\end{equation}
where $ \gamma $ is the Euler-Masceroni constant.
Then there exist $ \epsilon > 0 $ and $ \const > 0 $ such that for every $ \beta\in \C $ with $ |\beta|<\epsilon $ the following assertions hold.
\begin{enumerate}
   \item If $ \beta $ is such that $ \funcU(\beta) $ satisfies
   \begin{equation}\label{cond:(i)}
   \begin{aligned}
   \Re\funcU(\beta) &> \const|{\Im\funcU(\beta)}|^{3}
       & \text{ if } d&=1,
       \\
   \Re\funcU(\beta) &> 2|{\Im\funcU(\beta)}|^{\frac12} \big( 1 + \const |{\Re\funcU(\beta)}|\big)
    & \text{ if } d&=2,
   \end{aligned}
   \end{equation}
   then $ H_\beta $ has exactly one discrete eigenvalue $ \lambda_\beta $ in $ \C \setminus [0, \infty) $. Moreover, this eigenvalue is
   simple and obeys as $\beta \to 0$
\begin{equation}\label{expansion:EV}
   \begin{aligned}
       \sqrt{-\lambda_\beta} & = \frac{U}{2} \beta
       - \frac{U_1}{2} \beta^2 + \calO( \beta^3 ) & \text{ if } d& =1,
       \\
       \log(-\lambda_\beta)
       & =
       -\frac{4\pi}{U\beta} - 4\pi \frac{U_1}{U^2} + \log(4) - 2 \gamma + \calO(\beta) & \text{ if } d & =2.
   \end{aligned}
\end{equation}
   \item If $ \beta $ is such that $ \funcU(\beta) $ satisfies
   \begin{equation}\label{cond:(ii)}
   \begin{aligned}
       \Re\funcU(\beta) &< -\const|{\Im\funcU(\beta)}|^{3}     & \text{ if } d&=1,
       \\
       \Re\funcU(\beta) &< 2|{\Im\funcU(\beta)}|^{\frac12} \big( 1 - \const |{\Re\funcU(\beta)}|\big)     & \text{ if } d&=2,
   \end{aligned}
   \end{equation}
   then $ H_\beta $ has no eigenvalue in $ \C \setminus [0, \infty) $.
\end{enumerate}
\end{theorem}

\begin{remark}\label{rem3}
(i) Observe that the existence of weakly coupled eigenvalues is a purely one and two-dimensional effect.
Indeed, for $d \geq 3$ and $V\in L^{d/2}(\R^d)$ it follows from \eqref{EV.abs.cond} that
$\sigma_{\rm p}(H_\beta)= \emptyset$
for all $\beta \in \C$ such that
\begin{equation}\label{noEv}
D_{0,d} \int_{\R^d} |V(x)|^{\frac{d}{2}} \dx < \frac{1}{|\beta|^\frac d2}.
\end{equation}
Hence, if $ V \in L^1(\R^d) \cap L^p(\R^d) $ with $ p \geq d/2 $ (cf.~\eqref{p.rel.com}),
then $ V \in L^{d/2}(\R^d) $ and no
weak coupling phenomenon in the spirit of Theorem~\ref{theorem:s=2} occurs. 

(ii) In dimension $ d = 1,2 $ it follows from the spectral enclosures \eqref{ineq:Davies}
and \eqref{ineq:Frank} that weakly coupled eigenvalues of $ H_\beta $ (if they exist) obey
\begin{equation}\label{spectral_inclusion.new}
|\lambda| = \calO(|\beta|^2), \quad \beta \to 0,
\end{equation}
which is consistent with the asymptotic expansion \eqref{expansion:EV}. In fact, in Example~\ref{example:s=2.V.real.theta} below we consider potentials where \eqref{ineq:Davies} becomes sharp in the weak coupling regime.

(iii) We also note that in dimension $ d = 1,2 $ potentials $ V $ satisfying \eqref{assumption:family_v} do not induce 
positive eigenvalues, that is, $ \sigma_{\rm p}(H_\beta) \cap (0,\infty) = \emptyset $ for any $ \beta\in\C $;
cf.~\cite[Rem.~3.3]{Frank2017} for $ d = 1 $ and 
\cite{IonescuJerison2003, KochTataru2006} for $ d = 2 $.
\end{remark}

In the special case of a real-valued potential $V$ and complex coupling $\beta$, the conditions \eqref{cond:(i)} and \eqref{cond:(ii)} simplify.

\begin{corollary}  \label{cor:s=2.V.real}
Let the assumptions be as in Theorem~\ref{theorem:s=2} and, in addition, let
$ V : \R^d \to \R$ and $ U > 0 $.
Then there exist $ \epsilon' > 0 $ and $ \constcor > 0 $ such that for every $ \beta \in \C $ with $ |\beta| < \epsilon' $ the following assertions hold.
\begin{enumerate}
   \item If $ \beta $ satisfies
   \begin{equation} \label{cond:(i).cor.V.real}
   \begin{aligned}
       \Re\beta &> \left( -\frac{U_1}{U} + \constcor|{\Im\beta}| \right) (\Im\beta)^2 & \text{if } d=1,
       \\
       \Re\beta &> \frac{2}{U^\frac12} |{\Im\beta}|^{\frac12} \big( 1 + \constcor |{\Re\beta}| \big)  & \text{if } d=2,
   \end{aligned}
   \end{equation}
   then $ H_\beta $ has exactly one discrete eigenvalue $ \lambda_\beta $ in $ \C \setminus [0, \infty) $. Moreover, this eigenvalue is
   simple and obeys \eqref{expansion:EV}.
   \item If $ \beta $ satisfies
   \begin{equation} \label{cond:(ii).cor.V.real}
   \begin{aligned}
       \Re\beta
           &<
       \left( -\frac{U_1}{U} - \constcor|{\Im\beta}| \right) (\Im\beta)^2 & \text{if } d=1,
       \\
       \Re\beta
           &<
       \frac{2}{U^\frac12} |{\Im\beta}|^{\frac12} \big( 1 - \constcor |{\Re\beta}| \big) & \text{if } d=2,
   \end{aligned}
   \end{equation}
   then $ H_{\beta} $ has no eigenvalue in $ \C \setminus [0, \infty) $.
\end{enumerate}
\end{corollary}

The conditions and claims of Corollary~\ref{cor:s=2.V.real} are illustrated in Figure \ref{figure:1}.
\begin{figure}[H]
      \centering
        \begin{subfigure}[b]{\linewidth}
        \centering
            \begin{subfigure}[b]{0.43\linewidth}
                \includegraphics[width=\linewidth]{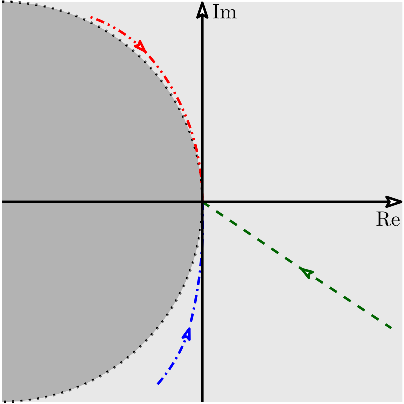}
               \caption{$\beta\to0$ in $d=1$}
               \label{W_r(d=1)}
           \end{subfigure}
           \hfill
      	   \centering
            \begin{subfigure}[b]{0.43\linewidth}
                \includegraphics[width=\linewidth]{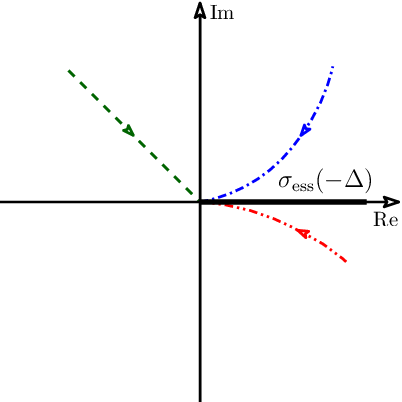}
                \caption{$\lambda_\beta\to0$ in $d=1$}
           \end{subfigure}
       \end{subfigure}

       \medskip

       \begin{subfigure}[b]{\linewidth}
        \centering
            \begin{subfigure}[b]{0.43\linewidth}
                \includegraphics[width=\linewidth]{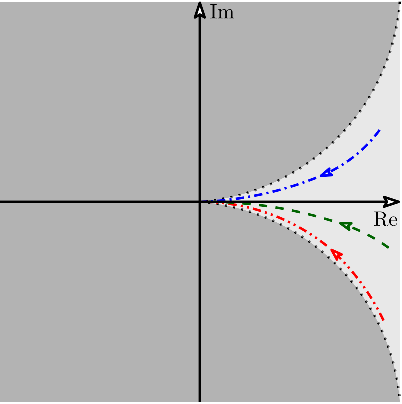}
               \caption{$\beta\to0$ in $d=2$}
               \label{W_r(d=2)}
           \end{subfigure}
            \hfill
            \centering
            \begin{subfigure}[b]{0.43\linewidth}
                \includegraphics[width=\linewidth]{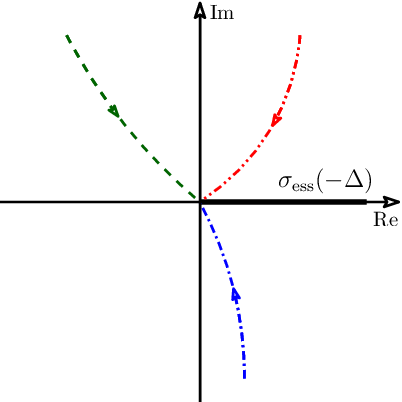}
               \caption{$\lambda_\beta\to0$ in $d=2$}
           \end{subfigure}
       \end{subfigure}
       \captionsetup{width=.82\linewidth}
		\caption{A schematic plot of Corollary \ref{cor:s=2.V.real} for $ U = U_1 = \constcor = 1 $.
In (a) and (c) the coupling parameter $\beta\rightarrow 0$ (along three different curves) leads to a simple discrete eigenvalue $\lambda_\beta$ of $H_\beta$ in (b) and (d).}
    \label{figure:1}
\end{figure}

\begin{example}  \label{example:s=2.V.real.theta}
Consider a non-negative potential $ V $, $ V \not\equiv 0 $, as in Corollary~\ref{cor:s=2.V.real} and assume that
the coupling is of the form
$ \beta(\epsilon) = e^{i\theta} \epsilon $ with $ \theta \in (-\pi, \pi] $ and $\epsilon>0$. Then one verifies (cf.~Figure \ref{figure:1}) that
there exists a discrete eigenvalue $ \lambda_{\beta(\epsilon)} $ as $ \epsilon\to 0+ $
of $ H_{\beta(\epsilon)} $ if and only if
\begin{align} \label{theta}
   \theta \in
   \begin{cases}
           \left[-\frac{\pi}{2}, \: \frac{\pi}{2} \right] & \text{ if } d = 1, \\
           \lbrace 0 \rbrace & \text{ if } d = 2.
   \end{cases}
\end{align}
If this eigenvalue exists it is simple, obeys \eqref{expansion:EV}, and is the only discrete eigenvalue 
of $ H_{\beta(\epsilon)} $ in the weak coupling limit. 
Note that in $ d = 1 $ the expansion \eqref{expansion:EV} yields
\begin{equation} \label{expansion:example}
   \sqrt{\smash{-\lambda_{\beta(\epsilon)}}\vphantom{\lambda_\beta}} = \frac{e^{i\theta}\epsilon}{2} \|V\|_{L^1} + \calO(\epsilon^2), \quad \epsilon \to 0+,
\end{equation}
which implies that the spectral enclosure \eqref{ineq:Davies} is sharp in the weak coupling regime; cf.~Section~\ref{subsec.example.proofs.1} for details.
\end{example}

In the next corollary we consider the special case where $ \beta $ is real-valued and $ V $
is a complex-valued potential.

\begin{corollary} \label{cor:s=2.V.complex.beta.real}
Let the assumptions be as in Theorem~\ref{theorem:s=2} and, in addition, let $\beta\in\R$.
Then for all sufficiently small $ \beta > 0 $ the following assertions hold.
\begin{enumerate}
   \item If
   \begin{equation} \label{cond:(i).cor.V.complex}
   \begin{aligned}
       &\Re U > 0 \text{ or alternatively } \Re U = 0 \text{ and } \Re U_1 < 0 &\text{if } d = 1, \\
       &\Im U = 0 \text{ and } \Re U > 2|{\Im U_1}|^\frac12 &\text{if } d = 2,
   \end{aligned}
   \end{equation}
   then $ H_\beta $ has exactly one discrete eigenvalue $ \lambda_\beta $ in $ \C \setminus [0, \infty) $. 
   Moreover, this eigenvalue is
   simple and obeys \eqref{expansion:EV}.
   \item If
   \begin{equation} \label{cond:(ii).cor.V.complex}
   \begin{aligned}
       &\Re U < 0 \text{ or alternatively } \Re U = 0 \text{ and } \Re U_1 > 0 &\text{if } d = 1, \\
       &\Im U \neq 0 \text{ or } \Re U < 2|{\Im U_1}|^\frac12 &\text{if } d = 2,
   \end{aligned}
   \end{equation}
   then $ H_{\beta} $ has has no eigenvalue in $ \C \setminus [0, \infty) $.
\end{enumerate}
\end{corollary}

\begin{example} \label{example_d=2}
For $ d = 1 $ the two conditions \eqref{cond:(i).cor.V.complex} and \eqref{cond:(ii).cor.V.complex} in
Corollary~\ref{cor:s=2.V.complex.beta.real} are easy to check and provide a convenient
characterization of the weak coupling phenomenon in one dimension.

Meanwhile, for $ d = 2 $ it is not immediately clear if there exist non-real potentials that satisfy \eqref{cond:(i).cor.V.complex}.
To address this issue, consider the family of potentials
\begin{align*}
   V(\alpha) = \alpha\Re V + i\Im V,\quad\alpha\in\R,
\end{align*}
where $ V $ is a fixed potential that satisfies \eqref{assumption:family_v} and \eqref{def:Rollnik_space_mainjussi}.
Assume, in addition, that
\begin{align*}
   \Im U = \int_{\R^2} \Im V \dx = 0
   \quad\text{ and }\quad
   \Re U = \int_{\R^2} \Re V \dx > 0.
\end{align*}
Then there exists $ \alpha^* \geq 0 $ (given by \eqref{alpha*}) such that
\begin{align*}
   V(\alpha) \text{ satisfies }
   \begin{cases}
       \eqref{cond:(i).cor.V.complex} & \text{ if } \alpha > \alpha^*, \\
       \eqref{cond:(ii).cor.V.complex} & \text{ if } \alpha < \alpha^*, \: \alpha \neq 0.
   \end{cases}
\end{align*}
Hence for $ \beta\to 0+ $ the Schrödinger operator $ H_\beta(\alpha)=-\Delta-\beta V(\alpha) $ in \eqref{H.beta.def}  has exactly one discrete eigenvalue (which is simple and obeys \eqref{expansion:EV}) if $ \alpha > \alpha^* $,
and has no discrete eigenvalues if $ \alpha < \alpha^* $, $ \alpha \neq 0 $; cf.~Section \ref{subsec.example.proofs.2} for 
details.
\end{example}

Let us now clarify the connection between our results 
and the classic weak coupling phenomenon for self-adjoint Schrödinger operators, going back to \cite{Simon1976}.
A potential $ V : \R^d \rightarrow \R $, $V\not\equiv 0$, satisfying \eqref{assumption:family_v} and the additional decay condition
\begin{equation*}
\begin{aligned}
\int_{\R^d} (1 + |x|^s)|V(x)| \dx & < \infty,
\end{aligned}
\end{equation*}
for $ s = 2 $ in $ d = 1 $ and for some $ s > 0 $ in $ d = 2 $,
leads to a negative eigenvalue of $ H_\beta $ for every $ \beta > 0 $ if and only if
$U\geq 0$. Moreover, this eigenvalue is simple, obeys \eqref{expansion:EV}, and is the only discrete eigenvalue of $ H_\beta $
as $ \beta\to 0+$. In turn, for $ V : \R^d \to \R $ the conditions of Corollary \ref{cor:s=2.V.complex.beta.real}
reduce to
\begin{align*}
	\eqref{cond:(i).cor.V.complex} \iff U > 0 \quad\text{and}\quad
	\eqref{cond:(ii).cor.V.complex} \iff U < 0.
\end{align*}
This means that for $ U \neq 0 $ the results of \cite{Simon1976} are contained as a special case
in our main theorem.
However, at the same time we emphasize that weak coupling
for complex-valued potentials is a much richer and diverse phenomenon.
Consider for simplicity
a non-negative potential $ V $, $ V \not\equiv 0 $,
with a coupling $ \beta\in \C $ as in Corollary~\ref{cor:s=2.V.real}.
Figure~\ref{figure:1} illustrates that, in dimension $d=1$, there exist arbitrarily small couplings
with $ \Re\beta < 0 $ that \emph{produce} a weakly coupled eigenvalue,
while in dimension $d=2$ there are couplings with $ \Re\beta \geq 0 $ for which
a similar effect \emph{cannot} be observed as $ |\beta|\to 0 $. 
Hence, a condition of the type $ \Re (U \beta) \geq 0 $, which may seem like the natural non-self-adjoint extension of 
Simon's conditions $ U \geq 0 $ and $ \beta > 0 $, is in general neither
necessary nor sufficient for the existence of weakly coupled eigenvalues in the non-self-adjoint
case. Finally, our results indicate that,
compared to the self-adjoint case, complex weak coupling is significantly less likely to occur in dimension $ d = 2 $ than in 
dimension $ d = 1 $; cf.~Figure~\ref{figure:1}.

For completeness, let us mention that the results in \cite{Simon1976} were further extended in
\cite{Klaus1977, KlausSimon1980, Rauch1980} for the self-adjoint case.
For a broader, yet non-exhaustive collection of related results we refer to 
\cite{Simon1977,Blankenbecler1977, Klaus1979, Lakaev1980, Patil1980, Klaus1982, Patil1982,
Holden1985, GesztesyHolden1987, BullaGesztesyRengerSimon1997, FassariKlaus1998, Melgaard2002}, \cite[Thm.~XIII.11, p.~336-338]{ReedSimonMethodsIV} and for more recent developments to 
\cite{FrankMorozovVugalter2011, KondejLotoreichik2014, ExnerKondejLotoreichik2018, CueninMerz2021, HoangHundertmarkRichterVugalter2023, MolchanovVainberg2023, ExnerKondejLotoreichik2024,Weidl-1999-24}.

\subsection{Strategy of the proof}

Our proof originates in the ideas in \cite{Simon1976}, which do not rely on variational arguments,
but instead exploit the criticality of $-\Delta$ in $d=1,2$ and the related singularity of the Green's function (see \eqref{def:resolvent}). The strategy of the proof is also inspired by \cite[Thm.~2.2]{CueninSiegl2018}, where a one-dimensional Dirac operator with non-self-adjoint potentials 
in a weak coupling setting is studied.
In particular, starting with the Birman-Schwinger principle (see Section~\ref{section:preliminaries}), the analysis of weakly coupled eigenvalues of $H_\beta$ in a certain subset $\calN_\beta$ of $\C$
is reduced to the study of zeros of an analytic function $\fbeta{d}$ (including the multiplicities), which can be asymptotically expanded as $\beta \to 0$, see  Proposition~\ref{prop:characterization_with_phi}. In this step, only the basic integrability conditions \eqref{assumption:family_v} are used. The additional condition \eqref{def:Rollnik_space_mainjussi} is employed in Section~\ref{section:improved} to characterize \emph{all} weakly coupled eigenvalues of $H_\beta$ as the zeros of $\fbeta{d}$ and also to expand $\fbeta{d}$ further, see Proposition~\ref{prop:equivalence}. In Section~\ref{section:main_results}, where the main results are proved, we use Lemma~\ref{lemma:rouche_eigenvalue_abstract}, which is
based on Rouch\'e's theorem, to show the existence of a simple zero and hence a simple eigenvalue of $H_\beta$. Moreover, the complete characterization of all weakly coupled eigenvalues of $H_\beta$ in Proposition~\ref{prop:equivalence} makes it possible 
to exclude further eigenvalues and to establish the non-existence of weakly coupled eigenvalues under a complementary condition \eqref{cond:(ii)}.

\subsection{Notations and conventions}
Throughout this paper, $ d \in \lbrace 1, 2 \rbrace $ denotes
the spatial
dimension. The scalar product (linear in the first entry) in $ L^2(\R^d) $ is denoted by $ \spf{\cdot}{\cdot} $; the standard norm on
$ L^p(\R^d) $ by $ \norm{\:\cdot\:}_{L^p} $. 
Furthermore, $ \C_+ = \lbrace z \in \C : \Re z > 0 \rbrace $ is the open right half-plane, and
$ D(z_0;r) = \lbrace z \in \C : |z-z_0| < r \rbrace $ denotes the open disc of radius
$ r > 0 $ centered at $ z_0 \in \C $.
We choose $ \sqrt{\cdot} $ as the branch of the square root with positive real part, and
$ \log : \C \setminus (-\infty, 0] \rightarrow \C $ is the principal
branch of the complex logarithm, that is, for $z=\vert z\vert e^{i\varphi}$, $\varphi\in (-\pi,\pi)$, we set
$\log(z)=\log\vert z\vert + i\varphi$; some useful properties of the logarithm can be found in the beginning of
Appendix~\ref{section:appenidx:inequalities}. Finally, we agree to use the letter $C$ for a positive constant in our estimates, which may change in
between the lines.

\subsection*{Acknowledgements}
This research was funded in part by the Austrian
Science Fund (FWF) 10.55776/P 33568-N. For the purpose of open access, the authors have applied a CC
BY public copyright licence to any Author Accepted Manuscript version arising from this submission.

\clearpage

\section{The Birman-Schwinger principle and eigenvalues in the weak coupling regime} \label{section:preliminaries}

In this section we consider the operators $ H_\beta $, $ \beta \in \C $, given by
\begin{equation}\label{H.beta.def2}
\begin{aligned}
H_\beta & = -\Delta - \beta V,
\\
\dom \Hbeta & = \lbrace f \in H^1(\R^d) : (-\Delta - \beta V)f \in L^2(\R^d) \rbrace,
\end{aligned}
\end{equation}
from \eqref{H.beta.def} in $L^2(\R^d)$ under the assumption that the (complex-valued) potential $V$ satisfies
\begin{equation} \label{assumption:family_v2}
	\begin{aligned}
		V \in
		\begin{cases}
			L^1(\R) & \text{ if } d = 1, \\
			L^1(\R^2) \cap L^{1+\eta}(\R^2) \text{ for some } \eta > 0 & \text{ if } d = 2;
		\end{cases}
	\end{aligned}
\end{equation}
cf.~\eqref{assumption:family_v}.
We recall that $H_\beta $ is $m$-sectorial and $\sigma_{\rm ess}(H_\beta) = [0,\infty)$, $\beta\in\C$ 
(for the various definitions of the essential spectrum we refer to \cite[Chap.~I.4]{EdmundsEvans1987} 
and for further details
to \cite[App.~B]{Frank2018}). In the following
we are interested in the discrete eigenvalues of $H_\beta$, typically in the situation that $\beta\in\C$ is sufficiently small.

\subsection{Birman-Schwinger principle} \label{subsec:general_BS}

Recall first that for $ z \in \C_+ $ the resolvent $ {(-\Delta + z^2)^{-1}} $ of the free Laplacian in $ L^2(\R^d) $
is an integral operator with the kernel
\begin{equation} \label{def:resolvent}
\begin{aligned}
	\ResolventKernel{d}(x,y;z)
		=
	\begin{cases}
		(2z)^{-1} e^{-z\abs[]{x-y}} & \text{ if } d = 1, \\
		(2\pi)^{-1} K_0(z\abs[]{x-y}) & \text{ if } d = 2,
	\end{cases}
	\qquad x,y \in \R^d,
\end{aligned}
\end{equation}
where $ K_0 : \C \setminus (-\infty,0] \rightarrow \C $ denotes the modified Bessel function of
second kind of order zero (for details
see Section~\ref{section:appenidx:inequalities}).
The eigenvalues $-z^2$ of $ \Hbeta $ outside $ [0, \infty) $ can be characterized via the Birman-Schwinger operator
$ Q(z) $, $z\in\C_+$, which is a compact integral operator in $ L^2(\R^d) $ with kernel
\begin{align} \label{def:birman_schwinger_operator_kernel}
	\KernelQ{d}(x,y;z)
		=
	|V(x)|^{\frac{1}{2}} \ResolventKernel{d}(x,y;z) V(y)^{\frac{1}{2}}, \quad
	x, y \in \R^d,
\end{align}
where we use the convention
\begin{align*}
	V(x)^{\frac{1}{2}}
		:=
	\begin{cases}
		|V(x)|^{-\frac{1}{2}}V(x) & \text{ if } V(x) \neq 0, \\
		0 & \text{ if } V(x) = 0.
	\end{cases}
\end{align*}
The Birman-Schwinger principle now reads as follows (for the definition of eigenvalues of operator families and eigenvalues of finite type, see Section~\ref{section:eigenvalues_operator_functions}). For a proof we
refer to \cite[Prop.~4.1]{Frank2018}.

\begin{lemma} \label{lemma:Birman_Schwinger}
Let $ \Hbeta = -\Delta - \beta V $ be as in \eqref{H.beta.def2} for $ V $ satisfying
\eqref{assumption:family_v2}. Let $ Q(z) $, $ z \in \C_+ $, be the compact integral operator
in $ L^2(\R^d) $ with kernel \eqref{def:birman_schwinger_operator_kernel}.
Then, for any $ z \in \C_+ $, there holds
\begin{align}
	\dim\kernel (\Hbeta + z^2) = \dim\kernel ( 1 - \beta Q(z) ).
\end{align}
Moreover, any eigenvalue of the family $ {1 - \beta Q} $ is an eigenvalue of finite type, and
the algebraic multiplicity of $ {1 - \beta Q} $ at $ z $ coincides with the algebraic multiplicity of $ -z^2 $ as an eigenvalue
of $ \Hbeta $.
\end{lemma}

\subsection{Reduction to zeros of a holomorphic function} \label{subsec:reduc}
The main idea to characterize the eigenvalues of $ \Hbeta $ in the weak coupling limit goes back \cite{Simon1976}, where the Birman-Schwinger operator $ Q(z) $ is decomposed in the form
\begin{align} \label{decomp:Qeps}
	Q(z) = L(z) + M(z),\quad z\in \C_+;
\end{align}
with integral operators $ L(z) $ and $ M(z) $ in $ L^2(\R^d) $ with respective kernels
\begin{align}
	\KernelL{d}(x,y; z) &:= |V(x)|^{\frac{1}{2}} \funcg{d}(z) V(y)^{\frac{1}{2}}, \nonumber \\
	\KernelM{d}(x,y; z) &:= |V(x)|^{\frac{1}{2}} \big( \ResolventKernel{d}(x,y;z) - \funcg{d}(z) \big) V(y)^{\frac{1}{2}}, \label{def:kernel:M}
\end{align}
and the function $ \funcg{d} : \C_+ \rightarrow \C $ is given by
\begin{equation} \label{def:g^d}
\begin{aligned}
	\funcg{d}(z)
		=
	\begin{cases}
		(2z)^{-1} & \text{if } d = 1, \\
		-(2\pi)^{-1} \log(z) & \text{if } d = 2.
	\end{cases}
\end{aligned}
\end{equation}
Note that $ \funcg{d} $ is injective and holomorphic on $ \C_+ $. Observe that $ L(z) $, $z\in \C_+$, is
a rank one operator and that $ \| L(\cdot) \| $ has a singularity of the same order as $ \funcg{d} $ at $ z = 0 $.
The usefulness of the decomposition \eqref{decomp:Qeps} is demonstrated in the following lemma.

\begin{lemma}	\label{lemma:characterization_eigenvalues_groundstate}
Let $ \Hbeta = -\Delta - \beta V $ be as in \eqref{H.beta.def2} for $ V $ satisfying
\eqref{assumption:family_v2}, let
$$
U = \int_{\R} V(x) \dx
$$
be as in \eqref{def:U_U_1_intro},
and let $ \funcg{d} : \C_+ \rightarrow \C $ be as in \eqref{def:g^d}.
Then for all sufficiently small $ \beta \in \C $ the set
\begin{align} \label{def:N_eps}
	\calN_\beta := \left\{ z \in \C_+ :  \| \beta M(z) \| < \frac{1}{2} \right\}
\end{align}
is non-empty and open, the function $\fbeta{d} : \calN_\beta \rightarrow \C $ given by
\begin{equation} \label{def:feps_lemma}
\begin{aligned}
	 \fbeta{d}(z) =  1 - \funcg{d}(z) \big[ U\beta + \beta
	 \spf[big]{[1 - \beta M(z)]^{-1}  \beta M(z) |V|^\frac{1}{2}}{\smash{\overline{V}^{\frac{1}{2}}}} \big],
	 \quad z \in \calN_\beta,
\end{aligned}
\end{equation}
is holomorphic on $ \calN_\beta $, and for all  $z \in \calN_\beta$ one has
\begin{align} \label{eq:eigenvalue_equation_lemma}
	-z^2 \in \sigma_{\rm{p}}(\Hbeta) \iff \fbeta{d}(z) = 0.
\end{align}
Moreover, if $ \fbeta{d}(z) = 0 $, then the algebraic multiplicity of $ -z^2 $ as an eigenvalue of $ \Hbeta $ coincides with the multiplicity of $ z $ as
a root of $ \fbeta{d} $.
\end{lemma}

\begin{proof}
Since $ M $ is an analytic family of bounded operators on $ \C_+ $, the set $ \calN_\beta $ is open and it is also non-empty for all
sufficiently small $ \beta \in \C $. Moreover, $ \fbeta{d} $ is well-defined for all $z \in \calN_\beta$ as $1-\beta M(z)$ is boundedly invertible;
the holomorphicity of $ \fbeta{d} $ follows from the analyticity of $ M $ and the holomorphicity of $ \funcg{d} $ on $ \C_+ $.

Let us continue by proving the equivalence \eqref{eq:eigenvalue_equation_lemma}.
By the Birman-Schwinger principle from Lemma~\ref{lemma:Birman_Schwinger},
it suffices to show that for all $z \in \calN_\beta$ the equivalence
\begin{align} \label{proof:ground_state_equiv_1}
	0 \in \sigma_{\rm{p}}(1 - \beta Q(z)) \iff  \fbeta{d}(z) = 0
\end{align}
is valid.
Since $1-\beta M(z)$ is boundedly invertible for $ z \in \calN_\beta $, the decomposition \eqref{decomp:Qeps} can be rewritten as
\begin{equation} \label{proof:ground_state_factorization_1}
\begin{aligned}
	1 - \beta Q(z)
		&=
	1 - \beta (L(z) + M(z)) \\
		&=
	(1 - \beta M(z))(1 - [1 - \beta M(z)]^{-1} \beta L(z))
\end{aligned}
\end{equation}
and this yields the equivalence
\begin{align} \label{proof:ground_state_equiv_2}
	0 \in \sigma_{\rm{p}}( 1 - \beta Q(z) )
		\iff
	0 \in \sigma_{\rm{p}}( 1 - [1 - \beta M(z)]^{-1} \beta L(z) ).
\end{align}
Next, let us set
\begin{align*}
	B :
	\begin{cases} L^2(\R^d) \rightarrow \C, \\ f \mapsto \int_{\R^d} V(y)^{\frac12} f(y) \dy, \end{cases}
	A_\beta(z) :
	\begin{cases} \C \rightarrow L^2(\R^d), \\ \varphi \mapsto [1 - \beta M(z)]^{-1} |V|^{\frac{1}{2}} \varphi.
	\end{cases}
\end{align*}
A short computation shows the operator identity
\begin{align} \label{proof:ground_state_identity_1}
	[1 - \beta M(z)]^{-1} \beta L(z) = \funcg{d}(z) \beta A_\beta(z) B;
\end{align}
indeed, for any $ f \in L^2(\R^d) $ one clearly has $ \funcg{d}(z) |V|^{\frac12} B f = L(z)f $ and hence
\begin{align*}
	\funcg{d}(z) \beta A_\beta(z) B f
		=
	\funcg{d}(z) \beta [ 1 - \beta M(z)]^{-1} |V|^{\frac{1}{2}} B f
		=
	[ 1 - \beta M(z)]^{-1} \beta L(z) f. 
\end{align*}
By combining \eqref{proof:ground_state_identity_1} with \eqref{proof:ground_state_equiv_2},
it is clear that
\begin{equation}
\begin{aligned} \label{proof:ground_state_equiv_3}
		0 \in \sigma_{\rm{p}}(1 - \beta Q(z))
		&\iff 0 \in \sigma_{\rm{p}}( 1 - [1 - \beta M(z)]^{-1} \beta L(z) ) \\
		&\iff 0 \in \sigma_{\rm{p}}( 1 - \funcg{d}(z) \beta A_\beta(z) B ) \\
		&\iff 0 \in \sigma_{\rm{p}}( 1 - \funcg{d}(z) \beta B A_\beta(z) ),
\end{aligned}
\end{equation}
where we used for the last equivalence  that
$0 \in \sigma_{\rm{p}}(1 - CD) \iff 0 \in \sigma_{\rm{p}}(1 - DC)$
for everywhere defined and bounded operators $ C $ and $ D $ between Hilbert spaces.
Finally, using the definition of $ A_\beta(z) $, $ B $, and $ U $ we obtain
\begin{align*}
	\beta B A_\beta(z)
		&=
	\beta \spf[big]{[1 - \beta M(z)]^{-1} |V|^\frac{1}{2}}{\smash{\overline{V}^{\frac{1}{2}}}} \\
		&=
	\beta \spf[big]{[1 - \beta M(z)]^{-1}
		(1 - \beta M(z) + \beta M(z) )|V|^\frac{1}{2}}{\smash{\overline{V}^{\frac{1}{2}}}} \\
		&=
	U \beta + \beta \spf[big]{[1 - \beta M(z)]^{-1}  \beta M(z) |V|^\frac{1}{2}}{\smash{\overline{V}^{\frac{1}{2}}}}.
\end{align*}
Hence
\begin{equation*}
		1 - \funcg{d}(z) \beta B A_\beta(z)
			=
		1 - \funcg{d}(z) \big[ U \beta + \beta \spf[big]{[1 - \beta M(z)]^{-1}  \beta M(z) |V|^\frac{1}{2}}{\smash{\overline{V}^{\frac{1}{2}}}} \big]	
			=
		\fbeta{d}(z),
\end{equation*}
which together with \eqref{proof:ground_state_equiv_3} shows
\eqref{proof:ground_state_equiv_1} and thus \eqref{eq:eigenvalue_equation_lemma}.

It remains to prove the claim about the algebraic multiplicities.
By the Birman-Schwinger principle from Lemma~\ref{lemma:Birman_Schwinger}, the algebraic multiplicity of
$ -z^2 \in \sigma_{\rm{p}}(\Hbeta)$ with $z \in \calN_\beta$ coincides with the algebraic multiplicity of the family $ {1 - \beta Q} $
at $z$, so the claim follows if we verify
\begin{align} \label{proof:ground_state_identity_2}
		m_a(z; 1 - \beta Q) = m_a(z; \fbeta{d}).
\end{align}
To justify \eqref{proof:ground_state_identity_2} recall that
$ 1 - \beta M(z) $ is boundedly invertible for all $ z \in \calN_\beta $.
Hence, we use the factorization 
\eqref{proof:ground_state_factorization_1}
and the fact that eigenvalues of operator families are invariant under multiplications by boundedly invertible
operators (see Lemma~\ref{lemma:invertible_operator_eigenvalues}) and conclude
\begin{align*}
	m_a(z; 1 - \beta Q)
		&=
	m_a(z; 1 - [1 - \beta M]^{-1} \beta L) .
\end{align*}
Finally, \eqref{proof:ground_state_identity_1} implies
\begin{align*}
	m_a(z; 1 - [1 - \beta M]^{-1} \beta L)
		&=
	m_a(z; 1 - \funcg{d} \beta A_\beta B) \\
		&=
	m_a(z; 1 - \funcg{d} \beta B A_\beta) \\
		&=
	m_a(z; \fbeta{d})
\end{align*}
where we used  Lemma~\ref{lemma:algebraic_multiplicity_A_B} to obtain the second equality.
The last quantity is the multiplicity of $ z $ as a root of $ \fbeta{d} $,  see Remark~\ref{remark:algebraic_multiplicity_root}.
\end{proof}

\subsection{Transformation of \texorpdfstring{$\fbeta{d}$}{Lambda}} \label{subsec:transform}
We will now transform the argument of the function $ \fbeta{d} $ in \eqref{def:feps_lemma}, which is useful
in our further analysis.
For this consider a restriction of the function $\funcg{d}$ in \eqref{def:g^d} given by
\begin{equation*}
\funcg{d}:\discBig{0}{\frac12} \cap \C_+\rightarrow \funcg{d}\Big( \discBig{0}{\frac12} \cap \C_+ \Big),
\quad z\mapsto
	\begin{cases}
		(2z)^{-1} & \text{if } d = 1, \\
		-(2\pi)^{-1} \log(z) & \text{if } d = 2,
	\end{cases}
\end{equation*}
and observe that $\funcg{d}$ is holomorphic on $ \disc{0}{\frac12} \cap \C_+ $ and bijective in the present context. 
The corresponding inverse has the form
\begin{equation*}
\funcg{d}^{-1}:\funcg{d}\Big( \discBig{0}{\frac12} \cap \C_+\Big)\rightarrow \discBig{0}{\frac12} \cap \C_+,
\quad w\mapsto
\begin{cases}
		\frac{1}{2w} & \text{ if } d = 1, \\
		\exp\left(- 2\pi w \right) & \text{ if } d = 2.
	\end{cases}
\end{equation*}
Next, we define the set
\begin{align} \label{def:Omega_Phi}
	\Omega :=
	\left\lbrace w \in \C_+ : \frac{1}{w} \in \funcg{d}\Big( \discBig{0}{\frac12} \cap \C_+ \Big) \right\rbrace,
\end{align}
and we note that $ \Omega$ is non-empty since $ \funcg{d}(\disc{0}{\frac12} \cap \C_+) \subset \C_+$.
Furthermore, we make use of the mapping
\begin{equation*}
\varphi:\Omega\rightarrow  \funcg{d}\Big( \discBig{0}{\frac12} \cap \C_+ \Big) ,\quad w\mapsto\frac{1}{w},
\end{equation*}
and we define
\begin{equation}\label{def:Phi}
\Phi:=\funcg{d}^{-1}\circ\varphi:\Omega\rightarrow \discBig{0}{\frac12} \cap \C_+,\quad
w\mapsto
\begin{cases}
		\frac{w}{2} & \text{ if } d = 1, \\
		\exp\left(-\frac{2\pi}{w}\right) & \text{ if } d = 2.
	\end{cases}
\end{equation}
From the construction it is clear that the function $\Phi:\Omega\rightarrow \disc{0}{\frac12} \cap \C_+$ is
well-defined, holomorphic, bijective, and one has $\funcg{d}(\Phi(w))=\frac{1}{w}$ for $w\in\Omega$.
In the next lemma we determine the set $\Omega$ explicitly.

\begin{lemma} \label{lemma:Phi}
The set $\Omega$ in \eqref{def:Omega_Phi} is given by
\begin{equation}\label{Omegajussi1}
\Omega=
       \disc{0}{1} \cap \C_+ \quad\text{if}\,\, d=1
\end{equation}
and
\begin{equation}\label{Omegajussi2}
\Omega=
       \left\lbrace w \in \C_+: \: \Re\left( \frac{1}{w} \right) > \frac{\log(2)}{2\pi}, \: \abs{\Im\left( \frac{1}{w} \right)}<\frac{1}{4} \right\rbrace \quad\text{if}\,\, d=2,
\end{equation}
and, in particular, $\Omega$ is bounded.
\end{lemma}

\begin{figure}[H]
      \centering
        \begin{subfigure}[b]{\linewidth}
        \centering
            \begin{subfigure}[b]{0.45\linewidth}
                \includegraphics[width=\linewidth]{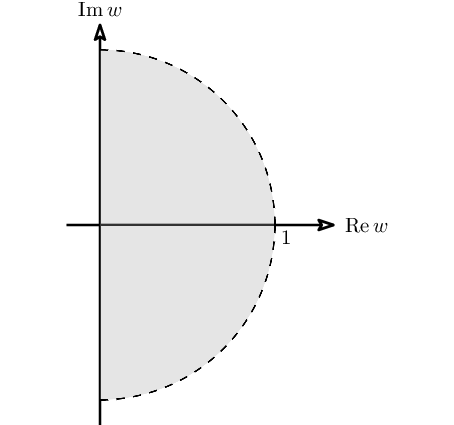}
               \caption{$\Omega$ in $ d = 1 $; cf.~\eqref{Omegajussi1}.}
           \end{subfigure}
           \hfill
      	   \centering
            \begin{subfigure}[b]{0.45\linewidth}
                \includegraphics[width=\linewidth]{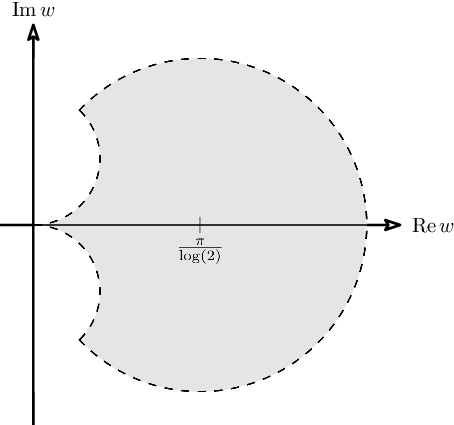}
                \caption{$\Omega$ in $ d = 2 $; cf.~\eqref{Omegajussi2} and \eqref{omega_alt_d=2}.}
           \end{subfigure}
       \end{subfigure}
	\caption{An illustration of the set $ \Omega $ defined in \eqref{def:Omega_Phi}.}
	\label{figure:2}
\end{figure}

\begin{proof}
For $d=1$, a direct computation shows
	\begin{align} \label{proof:Phi_funcg1}
		\funcg{1}\Big( \discBig{0}{\frac12} \cap \C_+ \Big)
		=
		\left\lbrace z \in \C_+ : |z| > 1 \right\rbrace
	\end{align}
	and hence we conclude \eqref{Omegajussi1}. For $d=2$, we obtain
	\begin{align} \label{proof:Phi_funcg2}
		\funcg{2}\Big( \discBig{0}{\frac12} \cap \C_+ \Big)
		=
		\left\lbrace z \in \C_+ : \Re z > \frac{\log(2)}{2\pi}, \: |{\Im z}|<\frac{1}{4} \right\rbrace,
	\end{align}
	which leads to \eqref{Omegajussi2}. To see that $\Omega$ is bounded in the case $d=2$ note that the
	condition 
	$$
	\frac{\Re w}{\vert w\vert^2}=\Re\left( \frac{1}{w} \right) > \frac{\log(2)}{2\pi},\quad w\in\Omega,
	$$
	implies $\vert w\vert^2 < \frac{2\pi}{\log(2)}\Re w\leq \frac{2\pi}{\log(2)}\vert w\vert$ and hence
	$\vert w\vert< \frac{2\pi}{\log(2)}$ for $w\in\Omega$.
	\end{proof}
	
	We also remark that for $ w \in \Omega $ in the case $d=2$ one has
	\begin{align*}
		\Re\left( \frac{1}{w} \right) > \frac{\log(2)}{2\pi}
		\iff
		\left( \Re w - \frac{\pi}{\log(2)} \right)^2 + (\Im w)^2 < \left( \frac{\pi}{\log(2)} \right)^2,
	\end{align*}
	thus leading to the alternative form (see Figure \ref{figure:2})
	\begin{equation} \label{omega_alt_d=2}
\Omega=D\left(\frac{\pi}{\log(2)}; \frac{\pi}{\log(2)}\right) \cap
       \left\lbrace w \in \C_+: \: \abs{\Im\left( \frac{1}{w} \right)}<\frac{1}{4} \right\rbrace.
\end{equation}

Now the map $ \Phi $ will be used to transform the function $ \fbeta{d}$ in \eqref{def:feps_lemma}
to the form \eqref{Lam.transf} below, where the last summand can be viewed as a perturbation term
that will be further expanded later.

\begin{proposition} \label{prop:characterization_with_phi}
Let $ \Hbeta = -\Delta - \beta V $ be as in \eqref{H.beta.def2} for $ V $ satisfying
\eqref{assumption:family_v2}. Let $ \fbeta{d} : \calN_\beta \rightarrow \C $ and $ \Phi : \Omega \rightarrow \disc{0}{\frac12} \cap \C_+ $
be given by \eqref{def:feps_lemma} and \eqref{def:Omega_Phi}, respectively.
Then for all sufficiently small $ \beta \in \C $, the following items hold.
\begin{enumerate}
	\item For all $w \in \Omega$ such that $\Phi(w) \in \calN_\beta$ we have
	\begin{equation}\label{Lam.Phi.equiv}
		\hspace{1cm}-\Phi(w)^2 \in \sigma_{\rm{p}}(H_\beta) \iff \fbeta{d}(\Phi(w)) = 0.
	\end{equation}
	If $w$ is a zero of $ \fbeta{d} \circ \Phi $, then its multiplicity as a root of $ \fbeta{d} \circ \Phi $ coincides with
	the algebraic multiplicity of the eigenvalue $ -\Phi(w)^2 \in \sigma_{\rm{p}}(\Hbeta) $.
	\item For all $w \in \Omega$ such that $\Phi(w) \in \calN_\beta$ we have
	\begin{equation}\label{Lam.transf}
		\hspace{1cm}\fbeta{d}(\Phi(w))
			=
		1 - \frac{U \beta}{w} - \frac{\beta }{w}
		\spf[big]{[1 - \beta M(\Phi(w))]^{-1}  \beta M(\Phi(w)) |V|^\frac{1}{2}}
		{\smash{\overline{V}^\frac12}}.
	\end{equation}
\end{enumerate}
\end{proposition}

\begin{proof} (i)
Lemma~\ref{lemma:characterization_eigenvalues_groundstate} implies for $ z \in \calN_\beta$
\begin{align*}
	-z^2 \in \sigma_{\rm{p}}(\Hbeta) \iff \fbeta{d}(z) = 0
\end{align*}
and the algebraic multiplicity of the eigenvalue $-z^2$ of $\Hbeta$ coincides with the multiplicity of $z$ as a root of
$\fbeta{d}$. Thus, if $z = \Phi(w) $ for some $ w \in \Omega $ such that $ \Phi(w) \in \calN_\beta $, then we obtain
\begin{align*}
	-\Phi(w)^2 \in \sigma_{\rm{p}}(H_\beta) \iff \fbeta{d}(\Phi(w)) = 0,
\end{align*}
in which case the algebraic multiplicity of $ -\Phi(w)^2 \in \sigma_{\rm{p}}(\Hbeta) $ coincides with the multiplicity of $ \Phi(w) $ as a root of $ \fbeta{d} $. Finally, since $ \Phi $ is bijective and holomorphic
it follows from Lemma~\ref{lemma:multiplicity} that the latter quantity is equal to the multiplicity of $ w $ as a
root of $ \fbeta{d} \circ \Phi $.

(ii) The form of $ \fbeta{d} \circ \Phi $ in \eqref{Lam.transf} follows directly from \eqref{def:feps_lemma} and the fact that, by construction,
$ \funcg{d}(\Phi(w)) = 1/w$; cf.~\eqref{def:Phi}. \qedhere
\end{proof}

\section{Characterization of all eigenvalues for Rollnik-type potentials} \label{section:improved}

In this section we continue investigating the discrete eigenvalues of the operators
$ H_\beta $, $ \beta \in \C $, in \eqref{H.beta.def2}. From
now on, for the rest of this paper, we assume, in addition to the condition \eqref{assumption:family_v2}, that the potential $ V : \R^d \rightarrow \C $ satisfies the Rollnik-type integrability condition
\begin{equation} \label{def:Rollnik_space_main2}
		\int_{\R^{2d}} \abs[]{V(x)} \big| \mu^{(d)}(x-y) \big|^2 \abs[]{V(y)} \dmath(x,y)<\infty;
\end{equation}
cf.~\eqref{def:Rollnik_space_mainjussi}.
For such potentials, it turns out in the next lemma that the operator family $M$ (which
are integral operators with kernels given by \eqref{def:kernel:M}) is uniformly bounded,
which then leads to a characterization of
\emph{all} discrete eigenvalues of $ \Hbeta $ as zeros of the function $ \fbeta{d} \circ \Phi $ (for all sufficiently small $\beta$), see Proposition~\ref{prop:equivalence}.

\begin{lemma} \label{lemma:M_eps_error_general}
	Let $ M(z) $, $ z \in \C_+ $, be the integral operator with kernel \eqref{def:kernel:M}.
	Suppose that $ V $ satisfies \eqref{assumption:family_v2} and \eqref{def:Rollnik_space_main2}. Then there exists a constant $ C_M > 0 $ such that
	\begin{equation}\label{M.bdd}
		\| M(z) \| \leq C_M
	\end{equation}
	holds for all $ z \in \disc{0}{\frac{1}{2}} \cap \C_+ $. In particular,
	\begin{equation}\label{D.Nb}
		\discBig{0}{\frac12} \cap \C_+ \subset \calN_\beta	
	\end{equation}
	for all sufficiently small $ \beta \in \C $.
\end{lemma}
	\begin{proof}
		By the definition of $ M(z) $ we have
		\begin{align*}
			\| M(z) \|^2
			\leq
			\int_{\R^{2d}} |V(x)| |\ResolventKernel{d}(x,y;z) - \funcg{d}(z)|^2 |V(y)| \dmath(x,y).
		\end{align*}
		Using Lemma~\ref{lemma:ineq:weber_kernels}(i)(a) and Lemma~\ref{lemma:ineq:weber_kernels}(ii)(a)
		we obtain (recall the we use the convention that $ C $ is a positive constant that may
		change in between estimates)
		\begin{align*}
			&|\ResolventKernel{1}(x,y;z) - \funcg{1}(z)|^2
			\leq
			\frac{1}{4}|x-y|^{2}
			=
			|\mu^{(1)}(x-y)|^2, \\
			&|\ResolventKernel{2}(x,y;z) - \funcg{2}(z))|^2
			\leq
			C(1 + \logabs{x-y}^2)
			\leq
			C(1 + |\mu^{(2)}(x-y)|^2),
		\end{align*}
		for $ d=1 $ and $ d=2 $, respectively. Hence, in both cases
		\begin{align*}
			\| M(z) \|^2
			&\leq
			C \| V \|_{L^1}^2 + C \int_{\R^{2d}} |V(x)| |\mu^{(d)}(x-y)|^2 |V(y)| \dmath(x,y)
			<
			\infty;
		\end{align*}
		cf.~\eqref{def:Rollnik_space_main2}.	
		Finally, \eqref{D.Nb} is an immediate consequence of \eqref{M.bdd}, see \eqref{def:N_eps}.
	\end{proof}

Recall the definitions of the complex numbers
\begin{equation}\label{ujussi}
	U := \int_{\R^d} V(x) \dx\quad\text{and}
   \quad
   U_1 := \int_{\R^{2d}} V(x) \mu^{(d)}(x-y) V(y) \dmath(x,y)
\end{equation}
from \eqref{def:U_U_1_intro}. Note that \eqref{def:Rollnik_space_main2} ensures that $U_1$ is well-defined since
the  inequality $ |{\mu^{(d)}(\cdot)}| \leq 1 + |{\mu^{(d)}(\cdot)}|^2 $ implies
\begin{equation} \label{U_1_exists}
\begin{aligned}
	\int_{\R^{2d}} &|V(x)| |\mu^{(d)}(x-y)| |V(y)| \dmath(x,y)  
		\\
	&\leq \| V \|_{L^1}^2
		+
	\int_{\R^{2d}} |V(x)| |\mu^{(d)}(x-y)|^2 |V(y)| \dmath(x,y) < \infty.
\end{aligned}
\end{equation}

Next, for $ U $ and $ U_1 $ as above let 
\begin{align} \label{def:a_eps}
	\funcU(\beta)
		=
	\begin{cases}
		U \beta - U_1 \beta^2 & \text{ if } d = 1, \\
		U \beta - \left[U_1 - (2\pi)^{-1}( \log(2) - \gamma) U^2 \right] \beta^2 & \text{ if } d = 2,
	\end{cases}
\end{align}
where $ \gamma $ denotes the Euler-Masceroni constant; cf.~\eqref{def:funcU_intro}.
It is helpful to notice that in the case $U \neq 0$ there exists a constant $m_\mathcal U \geq 1$ such that
\begin{equation}\label{Ucal.beta.equiv}
	\frac{1}{m_\mathcal U} |\funcU(\beta)| \leq |\beta| \leq m_\mathcal U |\funcU(\beta)|
\end{equation}
for all sufficiently small $\beta \in \C$.

The main result of this section is the following proposition.

\begin{proposition} \label{prop:equivalence}
Let $ \Hbeta = -\Delta - \beta V $ be as in \eqref{H.beta.def2} for $ V $ satisfying
\eqref{assumption:family_v2} and \eqref{def:Rollnik_space_main2}.
Let $ \fbeta{d} : \calN_\beta \rightarrow \C $ and $ {\Phi : \Omega \rightarrow \disc{0}{\frac12} \cap \C_+} $
be given by \eqref{def:feps_lemma} and \eqref{def:Phi}, respectively, and let $\funcU(\beta)$ be as in \eqref{def:a_eps}.
Then for all sufficiently small $ \beta \in \C $, the following items are true.
\begin{enumerate}
	\item There holds
	\begin{equation}\label{EV.equiv.main}
		\sigma_{\rm{p}}(\Hbeta) \setminus [0, \infty)
		=
		\left\lbrace -\Phi(w)^2 : w \in \Omega \text{ and } \fbeta{d}(\Phi(w)) = 0 \right\rbrace.
	\end{equation}
	If $ w $ is a root of $ \fbeta{d} \circ \Phi $, then its multiplicity as a root of $ \fbeta{d} \circ \Phi $ coincides with
	the algebraic multiplicity of the eigenvalue $ -\Phi(w)^2 \in \sigma_{\rm{p}}(\Hbeta) $.
	\item The function $ \fbeta{d} \circ \Phi $ admits the expansion
	\begin{align} \label{expansion:fbeta:Phi}
		\fbeta{d}(\Phi(w))
			=
		1 - \frac{1}{w} \left( \funcU(\beta) + r_\beta(w) \right),
		\quad w \in \Omega,
	\end{align}
	where $ r_\beta : \Omega \rightarrow \C $ is a holomorphic
	function that satisfies
	\begin{align} \label{error:rbeta:Phi}
		|r_\beta(w)| \leq C_r |\beta|^2 \max\left\lbrace |\beta|, |w| \right\rbrace,\quad w \in \Omega,
	\end{align}
	for some constant $ C_r > 0 $.
	\item Suppose that $ U \neq 0 $.
	Then there exists a constant $ C_\mathcal U > 0 $
	such that any root $ w \in \Omega $ of $ \fbeta{d} \circ \Phi $ satisfies
	\begin{align*}
		w \in \disc{\funcU(\beta)}{C_\mathcal U|\funcU(\beta)|^3}.
	\end{align*}
\end{enumerate}
\end{proposition}
\begin{proof}
(i) From Proposition~\ref{prop:characterization_with_phi} we already know that for all $w\in\Omega$
such that $\Phi(w) \in \calN_\beta$ one has
\begin{align*}
    -\Phi(w)^2 \in \sigma_{\rm{p}}(H_\beta) \iff \fbeta{d}(\Phi(w)) = 0,
\end{align*}
in which case the multiplicity of $ w $ as a root of $ \fbeta{d} \circ \Phi $ coincides with
the algebraic multiplicity of the eigenvalue $ -\Phi(w)^2 \in \sigma_{\rm{p}}(\Hbeta) $.
Thus, by Lemma~\ref{lemma:M_eps_error_general} the inclusion $(\supset)$ in \eqref{EV.equiv.main} is clear. For the remaining inclusion $(\subset)$ in \eqref{EV.equiv.main}
it suffices to show that for every eigenvalue $-z^2$ of $ \Hbeta $, $ z \in \C_+ $,
there exists
\begin{align*}
    w \in \Omega \text{ with } \Phi(w) \in \calN_\beta,
\end{align*}
such that $ \Phi(w) = z $. So let $ z \in \C_+ $ such that $ -z^2 \in \sigma_{\rm{p}}(\Hbeta) $,
then \eqref{spectral_inclusion.new} implies $ z \in \disc{0}{\frac12} \cap \C_+ $
for all sufficiently small $ \beta \in \C $.  Since
$ \Phi : \Omega \rightarrow \disc{0}{\frac12} \cap \C_+ $
is bijective, there exists $ w \in \Omega $ such that $ \Phi(w) = z $.
Finally,
we have \eqref{D.Nb}
for all sufficiently small $ \beta$, and hence
$ z = \Phi(w) \in \calN_\beta $.

(ii)  Fix $ \epsilon_0 > 0 $ such that for all
$ \beta \in \disc{0}{\epsilon_0} $ we have $ \disc{0}{\frac12} \cap \C_+ \subset \calN_\beta $;
note that the latter also implies
$ \Phi(w) \in \calN_\beta $ for $w\in\Omega$.
From Proposition~\ref{prop:characterization_with_phi}(ii) we have \eqref{Lam.transf} where we further expand $ [1 - \beta M(\Phi(w))]^{-1} $ in order to show that
\begin{align*}
	U \beta + \beta
	\spf[big]{[1 - \beta M(\Phi(w))]^{-1}
		\beta M(\Phi(w)) |V|^\frac{1}{2}}{\smash{\overline{V}^{\frac{1}{2}}}}
		=
	\funcU(\beta) + r_\beta(w),
\end{align*}
where $ \funcU(\beta) $ is as in \eqref{def:a_eps} and $ r_\beta $ has the claimed properties.
To simplify the notation, we set $ z = \Phi(w) $; then
\begin{equation}
\begin{aligned} \label{proof:decomp_A_eq1}
	&U \beta  + \beta\spf[big]{[1 - \beta M(z)]^{-1}
		\beta M(z) |V|^\frac{1}{2}}{\smash{\overline{V}^{\frac{1}{2}}}}  \\
	&\hspace{0.5cm}=
	U \beta  + \beta\spf[big]{[1 - \beta M(z)]^{-1}
		[1 - \beta M(z) + \beta M(z)] \beta M(z) |V|^\frac{1}{2}}{\smash{\overline{V}^{\frac{1}{2}}}} 	\\	
	&\hspace{0.5cm}=
		U \beta  + \beta^2 \spf[big]{M(z)|V|^{\frac{1}{2}}}{\smash{\overline{V}^{\frac{1}{2}}}} + \beta \calI_{\beta}(z),
\end{aligned}
\end{equation}
where
\begin{align} \label{proof:decomp:calI}
	\calI_{\beta}(z) := \spf[big]{[1 - \beta M(z)]^{-1} (\beta M(z))^2 |V|^{\frac12}}{\smash{\overline{V}^{\frac{1}{2}}}}.
\end{align}
In order to expand \eqref{proof:decomp_A_eq1} further, recall that $ M(z) $ is an integral operator in $ L^2(\R^d) $ with
integral kernel
\begin{align} \label{proof:decomp:kernel_M}
	\KernelM{d}(x,y;z) = |V(x)|^{\frac12} \big( \ResolventKernel{d}(x,y;z) - \funcg{d}(z) \big) V(y)^{\frac12}.
\end{align}
Thus, we find
\begin{equation} \label{proof:decomp_A_eq3}
\begin{aligned}
	\beta^2 \spf[big]{M(z)|V|^{\frac{1}{2}}}{\smash{\overline{V}^{\frac{1}{2}}}}
		&=
	\beta^2 \int_{\R^{2d}} V(x) \big( \ResolventKernel{d}(x,y;z) - \funcg{d}(z) \big) V(y) \dmath(x,y) \\
		&=
	\beta^2 \int_{\R^{2d}} V(x) \big( \funch{d}(x,y;z) - \funcg{d}(z) \big) V(y) \dmath(x,y) \\
		&\qquad\qquad + \beta^2 \errorR{d}(z),
\end{aligned}
\end{equation}
where
\begin{align} \label{proof:decomp:errorR}
	\errorR{d}(z)
		:=
	\int_{\R^{2d}} V(x) \left( \ResolventKernel{d	}(x,y;z) - \funch{d}(x,y;z) \right) V(y) \dmath(x,y)
\end{align}
for the function
\begin{align*}
	\funch{d}(x,y;z)
		=
	\begin{cases}
		(2z)^{-1} \left( 1 - z\abs[]{x-y} \right) & \text{ if } d = 1, \\
		-(2\pi)^{-1} \left( \log(z|x-y|) - \log(2) + \gamma \right) & \text{ if } d = 2.
	\end{cases}
\end{align*}
Next, we see that the definition \eqref{def:g^d} of $ \funcg{d} $ implies
\begin{equation}
\begin{aligned}
	\funch{d}(x,y;z) - \funcg{d}(z)
		&=
	\begin{cases}
		-2^{-1} |x-y| & \text{ if } d = 1, \\
		-(2\pi)^{-1} \left( \log|x-y| - \log(2) + \gamma \right) & \text{ if } d = 2,
	\end{cases} \\
		&=
	\begin{cases}
		-\mu^{(1)}(x-y) & \text{ if } d = 1, \\
		-\left[ \mu^{(2)}(x-y) - (2\pi)^{-1}\left( \log(2) - \gamma\right) \right] & \text{ if } d = 2,
	\end{cases}
\end{aligned}
\end{equation}
which together with the definition \eqref{ujussi} of $ U_1 $ and
\eqref{def:a_eps} of $ \funcU(\beta) $ shows
\begin{equation*}
\begin{aligned}
	U \beta + \beta^2 &\int_{\R^{2d}} V(x) \big( \funch{d}(x,y;z) - \funcg{d}(z) \big) V(y) \dmath(x,y)
		=
	\funcU(\beta) .
\end{aligned}
\end{equation*}
The last identity, combined with \eqref{proof:decomp_A_eq3}, yields
\begin{equation*}
\begin{aligned}
	U \beta  + \beta^2 \spf[big]{M(z)|V|^{\frac{1}{2}}}{\smash{\overline{V}^{\frac{1}{2}}}}
		=
	\funcU(\beta) + \beta^2 \errorR{d}(z),
\end{aligned}
\end{equation*}
so together with \eqref{proof:decomp_A_eq1} we finally arrive at
\begin{equation*}
\begin{aligned}
	U\beta  + \beta\spf[big]{[1 - \beta M(z)]^{-1}
		\beta M(z) |V|^\frac{1}{2}}{\smash{\overline{V}^{\frac{1}{2}}}} =
	\funcU(\beta) + \beta^2 \errorR{d}(z) + \beta \calI_{\beta}(z)
\end{aligned}
\end{equation*}
which means that the claimed expansion \eqref{expansion:fbeta:Phi}
holds with
\begin{align} \label{proof:decomp:r_beta}
	r_{\beta}(w) := \beta^2 \errorR{d}(\Phi(w)) + \beta \calI_{\beta}(\Phi(w)),
\end{align}
where $ \calI_{\beta} $ and $ \errorR{d} $ are given by \eqref{proof:decomp:calI} and
\eqref{proof:decomp:errorR}, respectively, and we
resubstituted $ z = \Phi(w) $ with $ w \in \Omega$.

We now show that the function $ r_\beta $ in \eqref{proof:decomp:r_beta} has the claimed
properties.
First, it is clear that $ r_\beta $ is holomorphic on $ \Omega $ as \eqref{expansion:fbeta:Phi} can 
be rearranged to 
\begin{align*}
	r_\beta(w) = w\left(1-\fbeta{d}(\Phi(w))\right) - \funcU(\beta)
\end{align*}
and both $ \fbeta{d} : \calN_\beta\to\C $ and $ \Phi : \Omega \to \disc{0}{\frac12} \cap \C_+ \subset \calN_\beta $ 
are holomorphic on their respective domains of definition for all sufficiently small $ \beta $.

Thus, it remains to show the claimed upper bound \eqref{error:rbeta:Phi}.
To this end, let $ w \in \Omega $, then \eqref{proof:decomp:r_beta} implies
\begin{align} \label{proof:decomp:r_beta_triangle_inequality}
	|r_\beta(w)|
		\leq
	|\beta|^2 |\errorR{d}(\Phi(w))| + |\beta| |\calI_{\beta}(\Phi(w))|.
\end{align}
We start by deriving an upper bound for the second term above.
Recall that $ \Phi(w) \in \disc{0}{\frac12} \cap \C_+ \subset \calN_\beta$  for all $ \beta \in \disc{0}{\epsilon_0} $, thus $ \| \beta M(\Phi(w)) \| < 1/2$.
The definition \eqref{proof:decomp:calI} of $ \calI_{\beta}(z) $, the Cauchy-Schwarz inequality and \eqref{M.bdd} imply
\begin{equation} \label{proof:decomp:estimate_second_term_1}
\begin{aligned}
	|{\beta}||\calI_{\beta}(\Phi(w))|
	&\leq
	|\beta|\| V^\frac12 \|_{L^2} \| [1 - \beta M(\Phi(w))]^{-1} \| \| \beta M(\Phi(w)) \|^2
	\| |V|^\frac12 \|_{L^2} \\
		&\leq
	2 \| V\|_{L^1} C_M^2 |\beta|^3.
\end{aligned}
\end{equation}
Next, we find an upper bound for $ \errorR{d}(\Phi(w)) $. In fact, we show below
\begin{align} \label{estimate_r:t}
	|\errorR{d}(z)| \leq C|\funcg{d}(z)|^{-1},
	\quad  z \in \discBig{0}{\frac12} \cap \C_+,
\end{align}
and
since $ \funcg{d}(\Phi(w)) = 1/w $, $ w \in \Omega $, the latter inequality
implies
\begin{align} \label{estimate_r_phi}
	|\errorR{d}(\Phi(w))| \leq C|w|,
	\quad
	 w \in \Omega.
\end{align}
Hence, together with \eqref{proof:decomp:r_beta_triangle_inequality} and
the estimate \eqref{proof:decomp:estimate_second_term_1} we find
\begin{align*}
	|r_\beta(w)|
		\leq
	C |\beta|^2 |w| + C |\beta|^3
	\leq
	2 C |\beta|^2 \max\left\lbrace |\beta|, |w| \right\rbrace,
\end{align*}
which leads to claim \eqref{error:rbeta:Phi}.

We establish the remaining estimate \eqref{estimate_r:t}. For $ d = 1 $, Lemma~\ref{lemma:ineq:weber_kernels}(i)(b) yields
\begin{align*}
	|\ResolventKernel{1}(x,y;z) - \funch{1}(x,y;z)|
		\leq
	\frac{1}{4}|z||x-y|^2
		=
	\frac{1}{2}|\mu^{(1)}(x-y)|^2 |\funcg{1}(z)|^{-1}
\end{align*}
where we used $ \funcg{1}(z) = (2z)^{-1} $ and the definition \eqref{def:mu_main}
of $ \mu^{(1)} $; it follows that
\begin{align*}
	|\errorR{1}(z)|
		\leq
	\frac{1}{2} |\funcg{1}(z)|^{-1} \int_{\R^2} |V(x)||\mu^{(1)}(x-y)|^2 |V(y)|\dmath(x,y)
		\leq
	C |\funcg{1}(z)|^{-1}
\end{align*}
which is the claimed estimate \eqref{estimate_r:t} for $d=1$.

Suppose next that $ d = 2 $. By introducing the regions
\begin{align} \label{proof:def:omega_1D_errorR}
	\Omega_1(z) &= \lbrace (x,y) \in \R^{4} : |x - y| < |z|^{-\frac12} \rbrace, \quad \Omega_2(z) = \Omega_1(z)^c,
\end{align}
it is clear that
\begin{align*}
	|\errorR{2}(z)|
		\leq
	\calJ_1^{(2)}(z) + \calJ_2^{(2)}(z)
\end{align*}
where
\begin{align*}
	\calJ_k^{(2)}(z) := \int_{\Omega_k(z)} |V(x)| |\ResolventKernel{2}(x,y;z) - \funch{2}(x,y;z)| |V(y)| \dmath(x,y), \quad k \in \lbrace 1, 2 \rbrace.
\end{align*}

We start by estimating $ \calJ_1^{(2)}(z) $. First, note that
for $ (x,y) \in \Omega_1(z) $ we can apply Lemma~\ref{lemma:ineq:weber_kernels}(ii)(c)
and obtain the inequality
\begin{align*}
	|\ResolventKernel{2}(x,y;z) - \funch{2}(x,y;z)|
		\leq
	C |z||{\log(z)}|
		=
	\frac{C |z||{\log(z)}|^2}{|{\log(z)}|}
		\leq
	C|\funcg{d}(z)|^{-1},
\end{align*}
where we used in the last estimate that $ \funcg{2}(z) = -(2\pi)^{-1}\log(z) $
and that the function $ \disc{0}{\frac12} \cap \C_+ \ni z \mapsto |z||{\log(z)}|^2 $
is bounded. In particular,
\begin{align*}
	\calJ_1^{(2)}(z)
		\leq
	C\norm{V}_{L^1(\R^2)}^2 |\funcg{d}(z)|^{-1}.
\end{align*}

Next, since $ z \in \disc{0}{\frac12}$, we obtain for all $ (x,y) \in \Omega_2(z) $ that
\begin{align} \label{proof:r_t_estimate_log}
	\sqrt{2} < |z|^{-\frac{1}{2}} \leq |x-y|
	\iff
	\frac{1}{2} \log(2) < \frac{1}{2} \logabs{z} \leq \log|x-y|,
\end{align}
and hence
\begin{align*}
	\frac{1}{\logabs{x-y}}
		\leq
	\frac{2}{\logabs{z}}
		\leq
	\frac{C}{|{\log(z)}|}		
		\leq
	C|\funcg{2}(z)|^{-1},
	\quad  (x,y) \in \Omega_2(z),
\end{align*}
where we used the inequality \eqref{ineq:log(z)} in the second
estimate.
The last inequality, together with Lemma~\ref{lemma:ineq:weber_kernels}(ii)(b) implies
\begin{align*}
	|\ResolventKernel{2}(x,y;z) - \funch{2}(x,y;z)|
		&\leq
	C \left(1 + \log|x-y| \right)
		\leq
	C \log|x-y|  \\
		&=
	\frac{C \logabs{x-y}^2}{\logabs{x-y}}
		\leq
	C \logabs{x-y}^2 |\funcg{2}(z)|^{-1}
\end{align*}
for all $ (x,y) \in \Omega_2(z) $, where we also used \eqref{proof:r_t_estimate_log}
for the second inequality. Thus
\begin{align*}
	\calJ_2^{(2)}(z)
		\leq
	\frac{C}{|\funcg{2}(z)|} \int_{\Omega_2(z)} |V(x)| \logabs{x-y}^2 |V(y)| \dmath(x,y)
		\leq
	C |\funcg{2}(z)|^{-1};
\end{align*}
cf.~\eqref{def:Rollnik_space_main2}.
By summing up the respective estimates for $ \calJ_1^{(2)}(z) $ and $ \calJ_2^{(2)}(z) $,
\eqref{estimate_r:t} follows.

(iii) Note that $ \fbeta{d}(\Phi(w)) = 0 $ and item (ii) together imply
\begin{align} \label{proof:properties_zero_assump}
	w = \funcU(\beta) + r_\beta(w).
\end{align}
The inequality \eqref{error:rbeta:Phi}  and the fact that $\Omega$ is bounded (see Lemma~\ref{lemma:Phi})
imply
\begin{equation}\label{w0.est}
	|w|
	\leq
	|\funcU(\beta)| + |r_\beta(w)|
		\leq
	|\funcU(\beta)| + C_r |\beta|^2 \max\left\lbrace |\beta|, |w| \right\rbrace
	\leq 	|\funcU(\beta)| + \widetilde C_r |\beta|^2
\end{equation}
for all sufficiently small $ \beta \in \C $.
Next, since $ U \neq 0 $ by assumption, by \eqref{Ucal.beta.equiv} and \eqref{w0.est} we obtain that for all sufficiently small $ \beta \in \C $
\begin{align*}
	|w|
		\leq
	|\funcU(\beta)|( 1 + \widetilde C_r m_\calU^2 |\funcU(\beta)|)
		\leq
	2|\funcU(\beta)|.
\end{align*}
The last inequality, reinserted into the estimate \eqref{error:rbeta:Phi} for $ r_\beta $ and \eqref{proof:properties_zero_assump} shows that
\begin{align*}
	|w - \funcU(\beta)|
		=
	|r_\beta(w)|
		\leq
	C_r |\beta|^2 \max\left\lbrace |\beta|, |w| \right\rbrace
		\leq
	C_\mathcal U |\funcU(\beta)|^3,
\end{align*}
where we again used \eqref{Ucal.beta.equiv} to estimate $ |\beta| $ against $ |\funcU(\beta)| $.
\qedhere
\end{proof}

\section{Proofs of the main results} \label{section:main_results}

In the upcoming proofs, many of the appearing estimates are only true if $ \beta $ is sufficiently small
and confined to some particular area in the complex plane. Thus, when we write
\begin{align} \label{convention:beta}
    \beta \to 0,
\end{align}
we actually mean that $ |\beta|<\epsilon $ for some $ \epsilon > 0 $ that might change from line to line, and that
$ \beta $ is confined to a particular region, such as \eqref{cond:(i)} or \eqref{cond:(ii)}; which one will be clear from the context.

\subsection{Proof of Theorem \ref{theorem:s=2}}
The next lemma is a useful technical ingredient in the
proof of Theorem~\ref{theorem:s=2}.

\begin{lemma} \label{lemma:Omega.inclusion}
Let $ \funcU(\beta) $, $ \beta\in \C $, be defined
as in \eqref{def:a_eps} and let $ \Omega $ be given by \eqref{def:Omega_Phi}.
Then for every $ K > 0 $ there exist $ \const > 0 $ and $ \epsilon > 0 $ such that for all
$ |\beta| < \epsilon $ the following assertions hold.
\begin{enumerate}
    \item If $ \beta $ satisfies \eqref{cond:(i)}, that is
    \begin{equation} \label{cond(i).lemma}
    \begin{aligned}
        \Re\funcU(\beta) &> \const|{\Im\funcU(\beta)}|^{3}
            & \text{ if } d&=1,
            \\
        \Re\funcU(\beta) &> 2|{\Im\funcU(\beta)}|^{\frac12} \big( 1 + \const |{\Re\funcU(\beta)}|\big)
         & \text{ if } d&=2,
    \end{aligned}
    \end{equation}
    then $ \disc{\funcU(\beta)}{K|\funcU(\beta)|^3} \subset \Omega $.
    \item If $ \beta $ satisfies \eqref{cond:(ii)}, that is
    \begin{equation} \label{cond(ii).lemma}
    \begin{aligned}
        \Re\funcU(\beta) &< -\const|{\Im\funcU(\beta)}|^{3}     & \text{ if } d&=1,
        \\
        \Re\funcU(\beta) &< 2|{\Im\funcU(\beta)}|^{\frac12} \big( 1 - \const |{\Re\funcU(\beta)}|\big)     & \text{ if } d&=2,
    \end{aligned}
    \end{equation}
    then $ \disc{\funcU(\beta)}{K|\funcU(\beta)|^3} \cap \Omega = \emptyset $.
\end{enumerate}
\end{lemma}

\begin{proof}
We start with the case $ d = 1 $, where
$ {\Omega = \disc{0}{1} \cap \C_+} $ according to \eqref{Omegajussi1}.
If $ w \in \disc{\funcU(\beta)}{K|\funcU(\beta)|^3} $
it is clear that $|w|<1$ as $ \beta \to 0 $. Hence, (i) and (ii) follow if we show that \eqref{cond(i).lemma} and
\eqref{cond(ii).lemma} with a sufficiently large $R>0$ imply $ \Re w > 0 $ and $ \Re w < 0 $ as $ \beta\to 0 $, respectively.
To this end note that $ w \in \disc{\funcU(\beta)}{K|\funcU(\beta)|^3} $ and the elementary inequality
\begin{align*}
    |a|^3 \leq 4|{\Re a}|^3 + 4|{\Im a}|^3, \quad a \in \C,
\end{align*}
imply
\begin{equation} \label{proof:s2:Re_d=1}
    \begin{aligned}
        \Re w
        &=
        \Re\funcU(\beta) + \Re(w - \funcU(\beta)) \\
        &=
        \Re\funcU(\beta) + \calO(|\funcU(\beta)|^3)
        \\
        &=
        \Re\funcU(\beta) + \calO([\Re\funcU(\beta)]^3) + \calO([\Im\funcU(\beta)]^3) \\
            &=
        \Re\funcU(\beta) \bigg(
        1 + \calO([\Re\funcU(\beta)]^2) + \calO\bigg( \frac{[\Im\funcU(\beta)]^3}{\Re\funcU(\beta)} \bigg)
        \bigg)
        \\
        &
        = \Re\funcU(\beta) \left(1 + \calO([\Re\funcU(\beta)]^2) + \frac1 \const \calO\left(1\right) \right),
        \quad\beta\to 0,
    \end{aligned}
\end{equation}
where the last asymptotic estimate follows after noticing that \eqref{cond(i).lemma} and \eqref{cond(ii).lemma} can be rearranged
to obtain $ |{\Im\funcU(\beta)}|^3 / |{\Re\funcU(\beta)}| < 1/\const $. Thus, \eqref{proof:s2:Re_d=1} implies that for $ \const > 0 $
sufficiently large
the numbers $ \Re w $ and $ \Re\funcU(\beta) $ have the same sign as $ \beta\to 0 $. Hence, as $ \beta\to 0 $ we have
$ \Re w > 0 $ if \eqref{cond(i).lemma} holds, i.e.,~$ w \in \Omega $, whereas \eqref{cond(ii).lemma} implies
$ \Re w < 0 $ and thus $ w \notin \Omega $; this proves (i) and (ii) for $ d = 1 $.

Let us continue with $ d = 2 $, where
\begin{align} \label{proof:Omega}
    \Omega
    =
    \left\lbrace w \in \C_+ : \Re\left( \frac{1}{w} \right) > \frac{\log(2)}{2\pi}, \: \abs{\Im\left( \frac{1}{w} \right)}<\frac{1}{4} \right\rbrace
\end{align}
by \eqref{Omegajussi2}. We start with some preliminary observations and inequalities.
Note first that $ w \in \disc{\funcU(\beta)}{K|\funcU(\beta)|^3} $ implies
\begin{equation}\label{proof:s2:ineq_1/w}
  \frac{c_1}{|\funcU(\beta)|} \leq \frac{1}{|w|} \leq \frac{c_2}{|\funcU(\beta)|}, \quad \beta \to 0,
\end{equation}
for some $c_1,c_2>0$,
and hence also
\begin{equation}\label{proof:s2:ineq_1/w-}
 \abs{\frac{1}{w} - \frac{1}{\funcU(\beta)}} = \abs{\frac{\funcU(\beta)-w}{w\funcU(\beta)}}
    \leq \frac{c_2 K |\funcU(\beta)|^3}{\vert \funcU(\beta)\vert^2}
    = \calO(|{\funcU(\beta)}|), \quad \beta \to 0.
\end{equation}
In particular, for $ w \in \disc{\funcU(\beta)}{K|{\funcU(\beta)}|^3} $ one verifies for $ \beta\to 0 $
\begin{equation} \label{Re(1/w)}
\begin{aligned}
    \Re\left( \frac{1}{w} \right)
        &=
    \Re\left( \frac{1}{\funcU(\beta)} \right) + \Re\left( \frac{1}{w} - \frac{1}{\funcU(\beta)} \right)
        =
    \frac{\Re\funcU(\beta)}{|{\funcU(\beta)}|^2} + \calO(|{\funcU(\beta)}|),
\end{aligned}
\end{equation}
and in a similar fashion for $ \beta\to 0 $
\begin{equation} \label{Im(1/w)}
    \abs{\Im\left( \frac{1}{w} \right)}
        =
    \frac{|{\Im\funcU(\beta)}|}{|{\funcU(\beta)}|^2} + \calO(|{\funcU(\beta)}|).
\end{equation}
As another preparatory step, we claim the following: Whenever
\begin{align} \label{ImRe2}
    |{\Im\funcU(\beta)}| < (\Re\funcU(\beta))^2,\quad\beta\to 0,
\end{align}
then we also have
\begin{equation} \label{Re=U}
\begin{aligned}
    |\funcU(\beta)| &= |{\Re\funcU(\beta)}| \left( 1+  \calO( [\Re\funcU(\beta)]^2) \right),\quad\beta\to 0, \\
    \frac{1}{|{\funcU(\beta)}|^2} &=
    \frac{1}{(\Re\funcU(\beta))^2} \left( 1+  \calO( [\Re\funcU(\beta)]^2) \right),\quad\beta\to 0.
\end{aligned}
\end{equation}
Indeed, note that \eqref{ImRe2} implies
\begin{equation*}
 |\funcU(\beta)|^2=(\Re\funcU(\beta))^2+(\Im\funcU(\beta))^2=(\Re\funcU(\beta))^2
 \bigl(1+ \calO( [\Re\funcU(\beta)]^2)\bigr), \quad\beta\to 0,
\end{equation*}
so the claim follows after taking both sides above to the power of
$ \frac12 $ and $ -1 $, respectively, and then applying the expansion
\begin{equation}\label{taylor}
 (1+x)^b=1+b x+\calO(x^2),\quad b\in\mathbb R,
\end{equation}
with $ b = \frac12 $ and $ b = -1 $.

To conclude our preliminary observations, it is useful to notice that whenever \eqref{ImRe2} (and hence \eqref{Re=U}) hold, the latter can be
inserted into \eqref{Re(1/w)} and \eqref{Im(1/w)} to obtain the expansions
\begin{equation} \label{Re(1/w)Re=U}
\begin{aligned}
    \Re\left(\frac{1}{w}\right)
        &=
    \frac{1}{\Re\funcU(\beta)}\left( 1 + \calO([\Re\funcU(\beta)]^2) \right) + \calO(\Re\funcU(\beta)) \\
        &=
    \frac{1}{\Re\funcU(\beta)} + \calO(\Re\funcU(\beta)), \quad\beta\to 0,
    \\
\end{aligned}
\end{equation}
and
\begin{equation} \label{Im(1/w)Re=U}
	\begin{aligned}
	\abs{\Im\left( \frac{1}{w} \right)}
		&=
		\frac{|{\Im\funcU(\beta)}|}{(\Re\funcU(\beta))^2}
		\left( 1+  \calO( [\Re\funcU(\beta)]^2) \right)
		+ \calO(\Re\funcU(\beta)) \\
		&=
		\frac{|{\Im\funcU(\beta)}|}{(\Re\funcU(\beta))^2}
		+ \calO(\Im\funcU(\beta)) + \calO(\Re\funcU(\beta)) \\
		&=
		\frac{|{\Im\funcU(\beta)}|}{(\Re\funcU(\beta))^2}
		+ \calO(\Re\funcU(\beta)),\quad\beta\to 0.
	\end{aligned}
\end{equation}

Let us continue by proving (i) for $ d = 2 $, i.e.,~$ w \in \Omega $ as
$ \beta\to 0 $ if \eqref{cond(i).lemma} holds
for some sufficiently large $ R > 0 $.
Rearranging \eqref{cond(i).lemma} yields
\begin{align*}
    |{\Im\funcU(\beta)}|^{\frac12}< \frac{1}{2}\, \frac{\Re\funcU(\beta)}{1 + \const |{\Re\funcU(\beta)}\vert}
     < \Re\funcU(\beta),
\end{align*}
so we have \eqref{ImRe2}, and thus can use \eqref{Re(1/w)Re=U} and \eqref{Im(1/w)Re=U}. 
Since $ \Re\funcU(\beta) > 0 $ by \eqref{cond(i).lemma}, the expansion \eqref{Re(1/w)Re=U} implies 
$ \Re(1/w) \to +\infty $ as $ \beta\to 0 $, i.e.,~$ w $ satisfies the first condition in
\eqref{proof:Omega}.

Applying our assumption \eqref{cond(i).lemma} to \eqref{Im(1/w)Re=U} and using \eqref{taylor} with $ b = -2 $ in the next step leads to the estimate
(recall that $ \Re\funcU(\beta) > 0 $)
\begin{equation} \label{Im(1/w)<1/4}
\begin{aligned}
    \abs{\Im\left( \frac{1}{w} \right)}
        &<
    \frac{1}{4(1 + \const\Re\funcU(\beta))^2} + \calO(\Re\funcU(\beta)) \\
    &=
    \frac{1}{4}\left(1 - 2 \const \Re\funcU(\beta) + \calO( [\Re\funcU(\beta)]^2)\right)
     + \calO(\Re\funcU(\beta)) \\
&=
\frac{1}{4} - \frac{1}{2}\const \Re\funcU(\beta)  + \calO(\Re\funcU(\beta)) \\
    &=
\frac{1}{4} - \frac{1}{2}\Re\funcU(\beta) \left( \const + \calO(1) \right),\quad
\quad\beta \to 0.
\end{aligned}
\end{equation}
So for $ R > 0 $ sufficiently large $ w $ also satisfies the second condition in
\eqref{proof:Omega}, i.e.,~$ w \in \Omega $; the proof of
(i) is complete for $ d = 2 $.

It remains to prove (ii) in the case $d=2$, i.e., we show $ w \notin \Omega $ as $ \beta\to 0 $ if \eqref{cond(ii).lemma} is satisfied for a sufficiently large $ \const > 0 $.
For this it suffices to consider the situation where \eqref{ImRe2} holds, as the
converse inequality implies
\begin{align*}
    |\funcU(\beta)|^2 = (\Re\funcU(\beta))^2 + (\Im\funcU(\beta))^2 \leq |{\Im\funcU(\beta)}| + (\Im\funcU(\beta))^2,
\end{align*}
and hence with \eqref{Im(1/w)}
\begin{align*}
    \abs{\Im\left( \frac{1}{w} \right)}
        &\geq
    \frac{1}{1 + |{\Im\funcU(\beta)}|} + \calO(|{\funcU(\beta)}|)
    \to 1,
    \quad\beta\to 0,
\end{align*}
i.e.,~$ |{\Im(1/w)}| \geq \frac14 $ and hence $ w \notin \Omega $.

So suppose that \eqref{cond(ii).lemma} and \eqref{ImRe2} hold, in which case we may
use \eqref{Re(1/w)Re=U} and \eqref{Im(1/w)Re=U}. Note that \eqref{ImRe2} also ensures that 
$ \Re\funcU(\beta) \neq 0 $.
If $ \Re\funcU(\beta) < 0 $, then \eqref{Re(1/w)Re=U}
implies $ \Re(1/w) \to -\infty $ as $ \beta\to 0 $, that is $ w \notin \Omega $.
If $ \Re\funcU(\beta) > 0 $, then
\eqref{Im(1/w)Re=U} and the assumption \eqref{cond(ii).lemma} imply
\begin{equation*}
\begin{aligned}
\abs{\Im\left(\frac{1}{w}\right)}
    &>
\frac{1}{4(1 - \const\Re\funcU(\beta))^2} + \calO(\Re\funcU(\beta)) \\
    &=
\frac{1}{4}\left(1 + 2 \const \Re\funcU(\beta) + \calO( [\Re\funcU(\beta)]^2)\right)
+ \calO(\Re\funcU(\beta)) \\
    &=
\frac{1}{4} + \frac{1}{2}\Re\funcU(\beta) \left( \const + \calO(1) \right),\quad
\quad \beta \to 0,
\end{aligned}
\end{equation*}
where the second line follows from \eqref{taylor} with $ b = - 2 $;
cf.~\eqref{Im(1/w)<1/4}. Therefore $ \Im(1/w) \geq \frac{1}{4} $ as $ \beta\to 0 $ if
$ \const > 0 $ is chosen
sufficiently large, i.e.,~$ w \notin\Omega $, and the proof of (ii) for $ d = 2 $ is complete.
\end{proof}

We can now use Lemma~\ref{lemma:Omega.inclusion} and Lemma~\ref{lemma:rouche_eigenvalue_abstract}
to give a proof of Theorem~\ref{theorem:s=2}.

\subsubsection{Proof of \upshape{(i)}}

According to Proposition~\ref{prop:equivalence}(i) it suffices to prove the existence of
$ R > 0 $ such that $ {\fbeta{d} \circ \Phi} $ has exactly one simple zero for $ \beta\to 0 $
satisfying \eqref{cond:(i)}.
To this end, we use the expansion \eqref{expansion:fbeta:Phi} and Lemma~\ref{lemma:rouche_eigenvalue_abstract} with
\begin{align} \label{proof:s2:def_a_b_delta}
    a = \funcU(\beta),
    \quad
    b(w) = r_\beta(w),
    \quad
    \delta = K |\funcU(\beta)|^2,\quad\calO_a=\disc{\funcU(\beta)}{4K|\funcU(\beta)|^3},
\end{align}
where the constant $ K > 0 $ is sufficiently large and will be specified below. Notice that $\delta \to 0$ as $\beta \to 0$ and so Lemma~\ref{lemma:rouche_eigenvalue_abstract} will be applied for all sufficiently small $\beta$.

Choose $ R > 0 $ such that $ \disc{\funcU(\beta)}{4K|\funcU(\beta)|^3} \subset \Omega $
for $ \beta\to 0 $
satisfying \eqref{cond:(i)},
which is possible by Lemma~\ref{lemma:Omega.inclusion}(i). Assumption (i) in Lemma~\ref{lemma:rouche_eigenvalue_abstract} 
is obviously satisfied and
Assumption (ii) in Lemma~\ref{lemma:rouche_eigenvalue_abstract} holds since
$ r_\beta $ is holomorphic on $ \Omega $, and hence on
$\disc{\funcU(\beta)}{4K |\funcU(\beta)|^3} \subset \Omega $ for $\beta\rightarrow 0$.
Next, to verify assumption (iii) in Lemma~\ref{lemma:rouche_eigenvalue_abstract},
consider $ w \in \disc{\funcU(\beta)}{4K |\funcU(\beta)|^3} $ and use \eqref{error:rbeta:Phi} and
\eqref{Ucal.beta.equiv} (note that $m_\mathcal U\geq1$) to estimate
\begin{equation} \label{proof:s2:error_1}
    \begin{aligned}
        |r_\beta(w)|
        &\leq
        C_r |\beta|^2 \max\left\{ |\beta|,  |w| \right\} \\
        &\leq
        C_r m_\mathcal U^2 |\funcU(\beta)|^2 \max\left\{ m_\mathcal U|\funcU(\beta)|,  |\funcU(\beta)| + 4K |\funcU(\beta)|^3 \right\} \\
        &\leq
        C_r m_\mathcal U^3|\funcU(\beta)|^3 \left( 1 + 4K |\funcU(\beta)|^2 \right) .
    \end{aligned}
\end{equation}
Thus, by selecting $ K > 2 C_r m_\mathcal U^3 $ and noting that $1 + 4 K |\funcU(\beta)|^2 < 2$ for $\beta\rightarrow 0$ we conclude from \eqref{proof:s2:error_1} that
\begin{equation*}
    \begin{aligned}
        |r_\beta(w)|
        \leq
        2 C_r m_\mathcal U^3 |\funcU(\beta)|^3
        <
        K |\funcU(\beta)|^3, \quad \beta \to 0. 
    \end{aligned}
\end{equation*}
Hence, assumption (iii) in Lemma~\ref{lemma:rouche_eigenvalue_abstract} is satisfied and therefore
$ {\fbeta{d} \circ \Phi} $ has exactly one simple root $w_\beta$ in the disc $\disc{\funcU(\beta)}{K |\funcU(\beta)|^3}$.
Using Proposition~\ref{prop:equivalence}(i) it follows that $ \lambda_\beta = - \Phi(w_\beta)^2 $ is a simple eigenvalue
of $H_\beta$ for $\beta\rightarrow 0$.

Next we show that $ \lambda_\beta $ is the only discrete eigenvalues of
$ H_\beta $ as $ \beta\to 0 $.
For this it suffices to check that $ w_\beta $ is the only zero of $ {\fbeta{d}\circ\Phi} $ as $ \beta\to 0 $, according to 
Proposition~\ref{prop:equivalence}(i).
In fact, if $ \widetilde{w} \in \Omega$ is a root of
$ \fbeta{d} \circ \Phi $, then necessarily $\widetilde{w} \in \disc{\funcU(\beta)}{C_\mathcal{U} |\funcU(\beta)|^3}$ as $ \beta\rightarrow 0 $
by Proposition~\ref{prop:equivalence}(iii). However, selecting
$ K > \max\{ 2 C_r m_\mathcal U^3, C_\mathcal U \} $ above, we conclude that there exists  exactly one root in the disc
$ \disc{\funcU(\beta)}{K|\funcU(\beta)|^3} $, that is, $\widetilde w = w_\beta$.

It remains to verify the asymptotic expansion of $ \lambda_\beta = - \Phi(w_\beta)^2 $ claimed in \eqref{expansion:EV}.
Recall that we have shown above that the root $w_\beta$ of $ \fbeta{d} \circ \Phi $ lies in
$\disc{\funcU(\beta)}{K |\funcU(\beta)|^3}$ as $\beta \to 0$, hence by \eqref{Ucal.beta.equiv} we obtain
\begin{align} \label{proof:s2:expansion_w}
    w_\beta = \funcU(\beta) + \calO(\funcU(\beta)^3)
    = \funcU(\beta) + \calO(\beta^3), \quad \beta \to 0.
\end{align}

Therefore, if $ d = 1 $, then we have $\funcU(\beta) = U \beta - U_1 \beta^2$ (see \eqref{def:a_eps}) and
thus by \eqref{def:Phi},
\begin{align*}
\sqrt{-\lambda_\beta}
=
\Phi(w_\beta)
=
\frac{1}{2} w_\beta
=
\frac{U}{2}\beta  - \frac{U_1}{2} \beta^2
+ \calO(\beta^3),
\quad\beta \rightarrow 0.
\end{align*}
Similarly, if $ d = 2 $, then $\funcU(\beta)=U \beta  - D\beta^2$, where $D=U_1-(2\pi)^{-1} (\log(2) - \gamma)U^2$
(see \eqref{def:a_eps}), and again using \eqref{proof:s2:expansion_w} and \eqref{def:Phi} we find
\begin{align*}
    \log(-\lambda_\beta)& =\log(\Phi(w_\beta)^2)
    =-\frac{4\pi}{w_\beta} 
    =
    -\frac{4\pi}{U \beta  - D \beta^2 + \calO(\beta^3)} 
    \\
    &  =
    -\frac{4\pi}{U \beta } \frac{1}{1 - (D/U)\beta + \calO(\beta^2)} 
    =
    -\frac{4\pi}{U \beta } \left[ 1 +  \frac{D}{U} \beta  + \calO(\beta^2) \right] 
    \\
    &=
    -\frac{4\pi}{U \beta } -   4 \pi \frac{D}{U^2} + \calO(\beta)  
    =
    -\frac{4\pi}{U \beta } - 4\pi \frac{U_1}{U^2} + \log(4) - 2 \gamma + \calO(\beta)
\end{align*}
as $\beta\rightarrow 0$. The proof of Theorem~\ref{theorem:s=2}(i) is complete.

\subsubsection{Proof of \upshape{(ii)}}

By Proposition~\ref{prop:equivalence}(i) it suffices to show that there exists $ \const > 0 $ such that $ {\fbeta{d} \circ \Phi} $ does not
possess any zeros in $ \Omega $ for $ \beta\to 0 $ satisfying \eqref{cond:(ii)}.

The latter assertion is an easy consequence of Proposition~\ref{prop:equivalence}(iii), which shows that any zero $ w \in \Omega $ of
$ \fbeta{d} \circ \Phi $ necessarily satisfies $ w \in \disc{\funcU(\beta)}{C_\calU |\funcU(\beta)|^3} $ and Lemma~\ref{lemma:Omega.inclusion}(ii), which
allows us to choose $ \const > 0 $ such that
for $ \beta\to 0 $ satisfying \eqref{cond:(ii)} there holds $ {\disc{\funcU(\beta)}{C_\calU |{\funcU(\beta)}|^3} \cap \Omega = \emptyset} $.
\qedhere

\subsection{Proof of Corollary \ref{cor:s=2.V.real} for \texorpdfstring{$\mathbf{d=1}$}{d=1}}

In the situation of Corollary $ \ref{cor:s=2.V.real} $ it assumed that the coupling $ \beta $ is complex and $ V $
is  real-valued. The assertions essentially follow from Theorem~\ref{theorem:s=2}.
We use the same notation $\beta \to 0$ as in the proof of Theorem~\ref{theorem:s=2}.

Recall first that for $d=1$ we have $\funcU(\beta) = U \beta - U_1 \beta^2$, where now $U>0$ and $U_1\in\mathbb R$ (see \eqref{def:U_U_1_intro}, \eqref{ujussi}), and hence
\begin{equation} \label{proof:sc:d1:Re_Im}
\begin{aligned}
    &\Re\funcU(\beta) = U \Re\beta - U_1 (\Re\beta)^2 + U_1 (\Im\beta)^2, \\
    &\Im\funcU(\beta) = U \Im\beta - 2 U_1 (\Re\beta) (\Im\beta);
\end{aligned}
\end{equation}
cf.~\eqref{def:funcU_intro}, \eqref{def:a_eps}.
Moreover, it
follows from \eqref{proof:sc:d1:Re_Im} that
\begin{align} \label{proof:sc:d1:ineq_Im}
        |{\Im\funcU(\beta)}|
        = U|{\Im\beta}| \left( 1 + \calO(\Re\beta) \right),
        \quad \beta \rightarrow 0.
\end{align}

\subsubsection{Proof of \upshape{(i)} for \texorpdfstring{$d=1$}{d=1}}
It suffices to show that there exists $R'>0$ such that
\begin{align} \label{proof:sc:d1:assump(i)}
        \Re\beta
            >
        \left( -\frac{U_1}{U} + \constcor|{\Im\beta}| \right) (\Im\beta)^2
\end{align}
implies
\begin{align} \label{proof:sc:d1:cond(i)}
\Re\funcU(\beta)
   >
    \const |{\Im\funcU(\beta)}|^3,
    \quad \beta\to 0,
\end{align}
with the constant $ \const >0 $ in Theorem~\ref{theorem:s=2}(i); then 
Theorem~\ref{theorem:s=2}(i) leads to the assertion.

If $ \Im\beta = 0 $ we see that for every $ \constcor > 0 $ one has
$\Re\beta > 0$ by 
\eqref{proof:sc:d1:assump(i)}, and $ \Im\funcU(\beta) = 0 $ by \eqref{proof:sc:d1:Re_Im} as well as 
\begin{equation} \label{proof:sc:d1:Im.zero}
\Re\funcU(\beta) = \Re\beta (U- U_1 \Re\beta)>0, \quad \beta \to 0,
\end{equation}
by \eqref{proof:sc:d1:Re_Im}. This leads to $ \Re\funcU(\beta) > 0 = R|{\Im\funcU(\beta)}|^3 $ as $\beta \to 0$, which is exactly
\eqref{proof:sc:d1:cond(i)}.

We assume further that $ \Im\beta \neq 0 $. For $\beta \to 0$ we then conclude from \eqref{proof:sc:d1:ineq_Im} that
\begin{align} \label{proof:sc:d1:Re/Im}
    \frac{\Re\funcU(\beta)}{|{\Im\funcU(\beta)}|^3}
        =
    \frac{\Re\funcU(\beta)}{U^3|{\Im\beta}|^3( 1 + \calO(\Re\beta))^3}
    =\frac{\Re\funcU(\beta)}{U^3|{\Im\beta}|^3}\bigl( 1 + \calO(\Re\beta)\bigr).
\end{align}
\par\noindent
\textit{Case 1:} $  U_1 \leq 0  $ \textit{or} $  \Re\beta \leq 0  $.
Note first that \eqref{proof:sc:d1:assump(i)} is equivalent to the inequality
\begin{align*}
    U \Re\beta + U_1 (\Im\beta)^2 > \constcor U|{\Im\beta}|^{3};
\end{align*}
the latter, when combined with \eqref{proof:sc:d1:Re_Im}
directly implies
\begin{align} \label{proof:sc:d1:ineq_1}
    \Re\funcU(\beta) > \constcor U|{\Im\beta}|^{3} - U_1(\Re\beta)^2.
\end{align}
Now on the one hand, if $ U_1 \leq 0 $, then \eqref{proof:sc:d1:Re/Im} and \eqref{proof:sc:d1:ineq_1} together show
\begin{equation*}
\begin{split}
    \frac{\Re\funcU(\beta)}{|{\Im\funcU(\beta)}|^3}
        &>
    \frac{\constcor U|{\Im\beta}|^{3} - U_1(\Re\beta)^2}{U^3|{\Im\beta}|^3} \bigl(1 + \calO(\Re\beta)\bigr)\\
        &\geq
    \frac{\constcor}{U^2} \bigl(1 +  \calO(\Re\beta)\bigr),
    \quad\beta\to 0.
\end{split}
\end{equation*}
Thus we can choose some sufficiently large $\constcor>0$ so that \eqref{proof:sc:d1:cond(i)} holds.

On the other hand, if $ \Re\beta \leq 0 $, then \eqref{proof:sc:d1:assump(i)} gives us
\begin{align*}
    |{\Re\beta}| < \left( \frac{U_1}{U} - \constcor |{\Im\beta}| \right) (\Im\beta)^2,
    \quad\beta\to 0,
\end{align*}
and hence $({\Re\beta})^2=\calO([\Im\beta]^4)$.
Therefore, with \eqref{proof:sc:d1:Re/Im} and \eqref{proof:sc:d1:ineq_1} we obtain
\begin{equation*}
\begin{split}
    \frac{\Re\funcU(\beta)}{|{\Im\funcU(\beta)}|^3}
        &>
    \frac{\constcor U|{\Im\beta}|^{3} - U_1(\Re\beta)^2}{U^3|{\Im\beta}|^3} \bigl(1 + \calO(\Re\beta)\bigr)\\
    &= \frac{\constcor U|{\Im\beta}|^{3} + \calO([\Im\beta]^4)}{U^3|{\Im\beta}|^3} \bigl(1 + \calO(\Re\beta)\bigr)\\
    &=
    \frac{\constcor}{U^2} \bigl(1 +  \calO(\Re\beta)\bigr) + \calO(\Im\beta)\bigl(1 + \calO(\Re\beta)\bigr)\\
    &=
    \frac{\constcor}{U^2} \bigl(1 +  \calO(|\beta|)\bigr),
    \quad\beta\to 0,
\end{split}
\end{equation*}
that is, \eqref{proof:sc:d1:cond(i)} holds if $\constcor>0$ is chosen sufficiently large in \eqref{proof:sc:d1:assump(i)}.
\\ \\
\noindent
\textit{Case 2: $ U_1 > 0 $ and $ \Re\beta > 0 $}.
Using $ U \Re\beta>U_1 (\Re\beta)^2$ for $\beta\to 0$
it follows from \eqref{proof:sc:d1:Re_Im} that
\begin{align*}
    \Re\funcU(\beta)
        >
    U_1(\Im\beta)^2,
    \quad\beta\to 0,
\end{align*}
so together with \eqref{proof:sc:d1:Re/Im}
\begin{align*}
    \frac{\Re\funcU(\beta)}{|{\Im\funcU(\beta)}|^3}
        >
    \frac{U_1}{U^3 |{\Im\beta}|} \bigl(1 + \calO(\Re\beta)\bigr)
        \rightarrow \infty,
    \quad
    \beta\rightarrow 0,
\end{align*}
which yields \eqref{proof:sc:d1:cond(i)} for any choice of $\constcor >0$ in \eqref{proof:sc:d1:assump(i)}.

\subsubsection{Proof of \upshape{(ii)} for \texorpdfstring{$d=1$}{d=1}}
By the same reasoning as in the proof of (i), it suffices to prove that there exists $\constcor>0$ such that 
\begin{align} \label{proof:sc:d1:assump(ii)}
    \Re\beta
        <
    \left( -\frac{U_1}{U} - \constcor|{\Im\beta}| \right) (\Im\beta)^2
\end{align}
implies
\begin{align} \label{proof:sc:d1:cond(ii)}
\Re\funcU(\beta)
        <
    -R |{\Im\funcU(\beta)}|^3,
    \quad\beta\to 0,
\end{align}
with the constant $ \const >0$ in Theorem~\ref{theorem:s=2}(ii).

If $ \Im\beta = 0 $ one clearly has for any $ \constcor > 0 $ that $\Re\beta < 0$ by 
\eqref{proof:sc:d1:assump(ii)},
$ \Im\funcU(\beta) = 0 $ by \eqref{proof:sc:d1:Re_Im}, and moreover
\begin{equation*}
\Re\funcU(\beta) = \Re\beta (U- U_1 \Re\beta)<0, \quad \beta \to 0;
\end{equation*}
cf.~\eqref{proof:sc:d1:Im.zero}. Hence, \eqref{proof:sc:d1:cond(ii)} is satisfied.

We assume further that $ \Im\beta \neq 0 $. For $\beta \to 0$ we again have \eqref{proof:sc:d1:Re/Im},
that is,
\begin{align} \label{proof:sc:d1:Re/Imjussi}
    \frac{\Re\funcU(\beta)}{|{\Im\funcU(\beta)}|^3}
    =\frac{\Re\funcU(\beta)}{U^3|{\Im\beta}|^3}\bigl( 1 + \calO(\Re\beta)\bigr).
\end{align}
\\
\noindent
\textit{Case 1:} $ U_1 \geq 0 $ \textit{or} $ \Re\beta \geq 0 $.
Observe that \eqref{proof:sc:d1:assump(ii)} can be rewritten in the form
\begin{align*}
    U \Re\beta + U_1 (\Im\beta)^2 < -\constcor U|{\Im\beta}|^{3},
\end{align*}
and combined with \eqref{proof:sc:d1:Re_Im} we find
\begin{align} \label{proof:sc:d1:ineq_4}
    \Re\funcU(\beta)
        <
    -\constcor U|{\Im\beta}|^{3} - U_1(\Re\beta)^2;
\end{align}
cf.~\eqref{proof:sc:d1:ineq_1}.
If $ U_1 \geq 0 $, then \eqref{proof:sc:d1:Re/Imjussi} and \eqref{proof:sc:d1:ineq_4}
yield
\begin{equation*}
\begin{split}
\frac{\Re\funcU(\beta)}{|{\Im\funcU(\beta)}|^3}
        &<
    \frac{-\constcor U|{\Im\beta}|^{3} - U_1(\Re\beta)^2}{U^3|{\Im\beta}|^3} \bigl( 1 + \calO(\Re\beta)\bigr)\\
        &\leq
    -\frac{\constcor}{U^2} \bigl( 1 + \calO(\Re\beta)\bigr) ,
    \quad\beta\to 0,
    \end{split}
\end{equation*}
and \eqref{proof:sc:d1:cond(ii)} follows if $ \constcor > 0 $ is chosen sufficiently large.
If $ \Re\beta \geq 0 $, then \eqref{proof:sc:d1:assump(ii)} gives us
\begin{align*}
    |{\Re\beta}|
        <
    \abs{-\frac{U_1}{U} - \constcor |{\Im\beta}|} (\Im\beta)^2,
    \quad\beta \rightarrow 0,
\end{align*}
and hence $(\Re\beta)^2=\calO([\Im\beta]^4)$. Using
\eqref{proof:sc:d1:Re/Imjussi} and \eqref{proof:sc:d1:ineq_4} we conclude
\begin{equation*}
    \begin{split}
    \frac{\Re\funcU(\beta)}{|{\Im\funcU(\beta)}|^3}
    &<
    \frac{-\constcor U|{\Im\beta}|^{3} - U_1(\Re\beta)^2}{U^3|{\Im\beta}|^3} \bigl( 1 + \calO(\Re\beta)\bigr)\\
        &=
    \frac{-\constcor U|{\Im\beta}|^{3} + \calO([\Im\beta]^4)}{U^3|{\Im\beta}|^3}\bigl( 1 + \calO(\Re\beta)\bigr)\\
        &=
    -\frac{\constcor}{U^2} \bigl( 1 + \calO(\Re\beta)\bigr)+ \calO(\Im\beta)\bigl( 1 + \calO(\Re\beta)\bigr)\\
    &=
    -\frac{\constcor}{U^2} \bigl( 1 + \calO(|\beta|)\bigr),
    \quad\beta\to 0,
\end{split}
    \end{equation*}
    that is, \eqref{proof:sc:d1:cond(ii)} holds if $\constcor>0$ is chosen sufficiently large in \eqref{proof:sc:d1:assump(ii)}.
\\ \\
\textit{Case 2:} $ U_1 < 0 $ and $ \Re\beta < 0 $.
Using $U \Re\beta<U_1 (\Re\beta)^2$ for $\beta\to 0$
it follows from \eqref{proof:sc:d1:Re_Im} that
\begin{align*}
    \Re\funcU(\beta)
        <
    -|{U_1}|(\Im\beta)^2,
    \quad\beta \to 0,
\end{align*}
and with \eqref{proof:sc:d1:Re/Imjussi} we conclude
\begin{align*}
    \frac{\Re\funcU(\beta)}{|{\Im\funcU(\beta)}|^3}
        <
    -\frac{|{U_1}|}{U^3|{\Im\beta}|}\bigl( 1 + \calO(\Re\beta)\bigr)\rightarrow -\infty,
    \quad\beta\to 0,
\end{align*}
which yields \eqref{proof:sc:d1:cond(ii)} for any choice of $\constcor >0$ in \eqref{proof:sc:d1:assump(ii)}.

\subsection{Proof of Corollary \ref{cor:s=2.V.real} for \texorpdfstring{$\mathbf{d=2}$}{d=2}}

In this case we have $\funcU(\beta) = U \beta - D \beta^2$, where $ D = [U_1 - (2\pi)^{-1}( \log(2) - \gamma) U^2] $ and
now $U>0$ and $U_1\in\mathbb R$ (see \eqref{def:U_U_1_intro}, \eqref{ujussi}). Hence
\begin{equation} \label{proof:sc:d2:Re_Im}
\begin{aligned}
    &\Re\funcU(\beta) = U \Re\beta - D (\Re\beta)^2 + D (\Im\beta)^2, \\
    &\Im\funcU(\beta) = U \Im\beta - 2 D (\Re\beta) (\Im\beta);
\end{aligned}
\end{equation}
cf.~\eqref{def:funcU_intro}, \eqref{def:a_eps}.
Hence, for $ \beta\to 0 $ one has
\begin{equation} \label{proof:sc:d2:ineq_Re_Im}
\begin{aligned}
    \Re\funcU(\beta) &= U \Re\beta (1 + \calO(\Re\beta)) + \calO([\Im\beta]^2),  \\
    |{\Im\funcU(\beta)}|^{\frac12} &= U^{\frac12}|{\Im\beta}|^{\frac12} \bigl(1 + \calO(\Re\beta)\bigr).
\end{aligned}
\end{equation}
It is also helpful to notice that whenever $ \Im\beta\neq 0 $ as $ \beta\to 0 $, then 
\eqref{proof:sc:d2:ineq_Re_Im} and the expansion \eqref{taylor} for $ b = -1 $ imply
\begin{equation} \label{proof:sc:d2:expansion}
\begin{aligned}
    \frac{\Re\funcU(\beta)}{2|{\Im\funcU(\beta)}|^{\frac12}}
        &=
    \frac{U \Re\beta (1 + \calO(\Re\beta)) + \calO([\Im\beta]^2)}{2 U^{\frac12}|{\Im\beta}|^{\frac12} \bigl(1 + \calO(\Re\beta)\bigr)} \\
        &=
    \frac{U^{\frac12} \Re\beta}{2|{\Im\beta}|^{\frac12}} \big( 1 + \calO(\Re\beta) \big) + \calO(|{\Im\beta}|^\frac32),
    \quad\beta\to 0.
\end{aligned}
\end{equation}

\subsubsection{Proof of \upshape{(i)} for \texorpdfstring{$d=2$}{d=2}}
It suffices to prove that there exists $R'>0$ such that
\begin{align} \label{proof:sc:d2:assumption(i)}
    \Re \beta
        >
    \frac{2}{U^\frac12} |{\Im\beta}|^{\frac12} \big( 1 + \constcor |{\Re\beta}| \big)
\end{align}
implies
\begin{align} \label{proof:sc:d2:cond(i)}
    \Re\funcU(\beta) > 2|{\Im\funcU(\beta)}|^{\frac12} \big( 1 + \const |{\Re\funcU(\beta)}|\big),
    \quad\beta\to 0,
\end{align}
where $ \const >0 $ is the constant in Theorem~\ref{theorem:s=2}(i).

If $ \Im\beta = 0 $, then for any $ \constcor > 0 $ one has $\Re \beta>0$ by \eqref{proof:sc:d2:assumption(i)},
$\Im\funcU(\beta)=0$ and
\begin{align} \label{proof:sc:d2:Imnonzero}
	\Re\funcU(\beta) = \Re\beta \left( U - D\Re\beta \right) > 0,
	\quad
	\beta\to 0,
\end{align}
by \eqref{proof:sc:d2:Re_Im}, i.e.,~\eqref{proof:sc:d2:cond(i)} holds.

We assume further that $ \Im\beta \neq 0 $. From \eqref{proof:sc:d2:assumption(i)} we find
\begin{equation*}
 |{\Im\beta}|^{\frac12}<\frac{U^\frac12 \Re\beta}{2( 1 + \constcor |{\Re\beta}|)}<U^\frac12 \Re\beta,
\end{equation*}
and hence
$ |{\Im\beta}|^2 = \calO([\Re\beta]^4) $ and $ |{\Im\beta}|^\frac32 = \calO([\Re\beta]^3) $ as $ \beta \to 0 $. With
\eqref{proof:sc:d2:ineq_Re_Im} it follows that
\begin{align} \label{proof:sc:d2:ineq_Re_1}
    \Re\funcU(\beta) = U \Re\beta \bigl(1 + \calO(\Re\beta)\bigr),
    \quad\beta\to 0.
\end{align}
In particular, with \eqref{proof:sc:d2:expansion} and \eqref{proof:sc:d2:assumption(i)}
\begin{align*}
    \frac{\Re\funcU(\beta)}{2|{\Im\funcU(\beta)}|^{\frac12}}
        &=
    \frac{U^{\frac12} \Re\beta}{2|{\Im\beta}|^{\frac12}} \big( 1 + \calO(\Re\beta) \big) + \calO(|{\Im\beta|}^\frac32) \\
        &>
    \big( 1 + \constcor|{\Re\beta}| \big) \big( 1+\calO(\Re\beta) \big) + \calO([\Re\beta]^3) \\
        &=
    1 + |{\Re\beta}|(\constcor + \calO(1)) \\
        &=
    1 + |{\Re\funcU(\beta)}|\left(\frac{\constcor}{U} + \calO(1)\right),
    \quad\beta\to 0,
\end{align*}
which for some sufficiently large $ \constcor > 0 $ yields \eqref{proof:sc:d2:cond(i)}.

\subsubsection{Proof of \upshape{(ii)} for \texorpdfstring{$d=2$}{d=2}}
We prove that
\begin{align} \label{proof:sc:d2:assumption(ii)}
    \Re\beta
        <
    \frac{2}{U^\frac12} |{\Im\beta}|^{\frac12} \big( 1 - \constcor |{\Re\beta}| \big)
\end{align}
for some sufficiently large $ \constcor > 0 $ implies
\begin{align} \label{proof:sc:d2:cond(ii)}
    \Re\funcU(\beta) < 2|{\Im\funcU(\beta)}|^{\frac12} \big( 1 - \const |{\Re\funcU(\beta)}|\big),
    \quad\beta\to 0,
\end{align}
where $ \const > 0 $ is as in Theorem~\ref{theorem:s=2}(ii).

If $ \Im\beta = 0 $, then for any $ \constcor > 0 $ we see that $\Re\beta<0$ by \eqref{proof:sc:d2:assumption(ii)},
$ \Im\funcU(\beta) = 0 $ by \eqref{proof:sc:d2:Re_Im}, and $ \Re\funcU(\beta) < 0 $ for $\beta\rightarrow 0$;
cf.~\eqref{proof:sc:d2:Imnonzero}. From this we conclude \eqref{proof:sc:d2:cond(ii)}.

From here on we assume $ \Im\beta \neq 0 $, and hence are allowed to use the expansion \eqref{proof:sc:d2:expansion}. 
If $|{\Im\beta}| \geq |{\Re\beta}| $ as $\beta \to 0$, then $ \Re\beta = \calO(\Im\beta) $ so
\eqref{proof:sc:d2:expansion} yields
\begin{equation*}
\begin{split}
  \frac{\Re\funcU(\beta)}{2|{\Im\funcU(\beta)}|^\frac12}
    = 
  \frac{U^\frac12 \Re \beta}{2|{\Im \beta}|^\frac12}(1+ \calO(\Re \beta)) + \calO(|{\Im \beta}|^\frac 32) 
  = \calO(|{\Im \beta}|^\frac 12),  \quad\beta\to 0,
\end{split}
\end{equation*}
i.e.,~\eqref{proof:sc:d2:cond(ii)} holds for any choice of $\constcor>0$ in \eqref{proof:sc:d2:assumption(ii)}.

Thus we further assume that $\Im \beta \neq 0$ and
\begin{align} \label{proof:sc:d2:Im<Re}
|{\Im\beta}| < |{\Re\beta}|,
\quad\beta \to 0.
\end{align}
By \eqref{proof:sc:d2:Im<Re} we have $|{\Im\beta}|=\calO(\Re\beta)$ and from \eqref{proof:sc:d2:ineq_Re_Im} we then obtain
\begin{align*} 
    \Re\funcU(\beta) = U \Re\beta (1 + \calO(\Re\beta)),
    \quad\beta\to 0,
\end{align*}
which together with \eqref{proof:sc:d2:expansion}, \eqref{proof:sc:d2:assumption(ii)} and \eqref{proof:sc:d2:Im<Re}
implies 
\begin{align*}
    \frac{\Re\funcU(\beta)}{2|{\Im\funcU(\beta)}|^{\frac12}}
        &=
    \frac{U^{\frac12} \Re\beta}{2|{\Im\beta}|^{\frac12}} \big( 1 + \calO(\Re\beta) \big) + \calO(|{\Im\beta|}^\frac32) \\
        &<
    \big( 1 - \constcor|{\Re\beta}| \big) \big( 1+\calO(\Re\beta) \big) + \calO(|{\Re\beta|}^\frac32) \\
        &=
    1 - |{\Re\beta}|(\constcor + \calO(1)) \\
        &=
    1 - |{\Re\funcU(\beta)}|\left( \frac{\constcor}{U} + \calO(1) \right) ,
    \quad\beta\to 0.
\end{align*}
For some sufficiently large $ \constcor > 0 $ this implies \eqref{proof:sc:d2:cond(ii)}.

\subsection{Proof of Corollary \ref{cor:s=2.V.complex.beta.real} for \texorpdfstring{$\mathbf{d=1}$}{d=1}}

In Corollary $ \ref{cor:s=2.V.complex.beta.real} $ it is assumed that the coupling constant $ \beta $ is positive and $ V $
is complex-valued. The assertions are consequences of Theorem~\ref{theorem:s=2}.

For $d=1$ we have $\funcU(\beta) = U \beta - U_1 \beta^2$, where now $\beta>0$
(see \eqref{def:funcU_intro}, \eqref{def:a_eps}), and hence 
\begin{equation} \label{pf:cor:Vcomplex:d1:ReIm}
\begin{aligned}
    \Re\funcU(\beta) &= \beta\Re U - \beta^2 \Re U_1, \\
    |{\Im\funcU(\beta)}| &= |{\beta\Im U - \beta^2 \Im U_1}| = \calO(\beta), \quad {\beta \rightarrow 0+}.
\end{aligned}
\end{equation}

\subsubsection{Proof of \upshape{(i)} for \texorpdfstring{$d=1$}{d=1}}
It follows from \eqref{pf:cor:Vcomplex:d1:ReIm} that for $\beta \rightarrow 0+$ we have $|{\Im\funcU(\beta)}|^3= \calO(\beta^3)$ and 
\begin{align*}
    \Re\funcU(\beta)
  =
    \begin{cases}
    \beta \Re U (1+\calO(\beta)) & \text{ if } \Re U > 0, \\
     \beta^2 |{\Re U_1}| & \text{ if } \Re U = 0 \text{ and } \Re U_1 < 0.
    \end{cases}
\end{align*}
Thus for $R>0$ from Theorem~\ref{theorem:s=2}(i) (in fact, for any $R>0$)
\begin{align} \label{pf:cor:Vcomplex:d1:cond(i)}
	\Re\funcU(\beta) > \const |{\Im\funcU(\beta)}|^3, \quad \beta\rightarrow 0+,
\end{align}
so the claim follows from Theorem~\ref{theorem:s=2}(i).

\subsubsection{Proof of \upshape{(ii)} for \texorpdfstring{$d=1$}{d=1}}
Similar to the proof of (i) it follows from \eqref{pf:cor:Vcomplex:d1:ReIm} that for $\beta \rightarrow 0+$
we have $|{\Im\funcU(\beta)}|^3= \calO(\beta^3)$ and
\begin{align*}
    \Re\funcU(\beta)
    =
    \begin{cases}
   	- \beta |{\Re U}| (1+\calO(\beta)) & \text{ if } \Re U < 0, \\
   - \beta^2 \Re U_1 & \text{ if } \Re U = 0 \text{ and } \Re U_1 > 0.
\end{cases}
\end{align*}
Hence, for $R>0$ from Theorem~\ref{theorem:s=2}(ii) we obtain
\begin{align*}
    \Re\funcU(\beta) < -\const |{\Im\funcU(\beta)}|^3, \quad \beta\rightarrow 0+,
\end{align*}
and the claim follows from Theorem~\ref{theorem:s=2}(ii).

\subsection{Proof of Corollary \ref{cor:s=2.V.complex.beta.real} for \texorpdfstring{$\mathbf{d=2}$}{d=2}}
Now we have $\funcU(\beta) = U \beta - D \beta^2$, where $ D = [U_1 - (2\pi)^{-1}( \log(2) - \gamma) U^2] $ and
$\beta>0$ (see \eqref{def:funcU_intro}, \eqref{def:a_eps}), and hence
\begin{equation} \label{pf:cor:Vcomplex:d2:ReIm}
\begin{aligned}
    \Re\funcU(\beta) &= \beta \Re U - \beta^2 \Re D = \beta \Re U + \calO(\beta^2) , \\
    |{\Im\funcU(\beta)}|^\frac12 &= |{\beta \Im U - \beta^2 \Im D}|^\frac12 = \beta^\frac12 |{\Im U - \beta \Im D}|^\frac12.
\end{aligned}
\end{equation}
For later purposes, we also note that
\begin{align}  \label{pf:cor:Vcomplex:d2:ImD}
    \Im D
        =
    \Im U_1 - \frac{1}{\pi}(\log (2) - \gamma)(\Re U)(\Im U),
\end{align}
and, in particular, $ \Im D = \Im U_1 $ if $ \Im U = 0 $.

\subsubsection{Proof of \upshape{(i)} for \texorpdfstring{$d=2$}{d=2}}
Using the assumption $ \Im U = 0 $ in \eqref{pf:cor:Vcomplex:d2:ReIm}
we obtain
\begin{equation} \label{pf:cor:Vcomplex:d2:eq1}
    |{\Im\funcU(\beta)}|^\frac12 = \beta^\frac12 |{\beta\Im D}|^\frac12 = \beta|{\Im U_1}|^\frac12,
\end{equation}
where also \eqref{pf:cor:Vcomplex:d2:ImD} was used in the last equality.

If $ \Im U_1 \neq 0 $, then \eqref{pf:cor:Vcomplex:d2:ReIm} and \eqref{pf:cor:Vcomplex:d2:eq1} imply
\begin{equation}\label{hurra77}
    \frac{\Re\funcU(\beta)}{2|{\Im\funcU(\beta)}|^\frac12}
        =
    \frac{\Re U + \calO(\beta)}{2|{\Im U_1}|^\frac12}
        \rightarrow
    \frac{\Re U}{2|{\Im U_1}|^\frac12}
        > 1,
    \quad \beta \rightarrow 0+,
\end{equation}
where the assumption $\Re U > 2|{\Im U_1}|^\frac12 $ was used in the last inequality.
With $R>0$ from Theorem~\ref{theorem:s=2}(i) we conclude from \eqref{hurra77} that
\begin{align}  \label{pf:cor:Vcomplex:d2:cond(i)}
    \Re\funcU(\beta) > 2|{\Im\funcU(\beta)}|^\frac12 \big( 1 + \const |{\Re\funcU(\beta)}| \big),
    \quad {\beta\to 0+},
\end{align}
and hence the claim follows from Theorem~\ref{theorem:s=2}(i).
If $ \Im U_1 = 0 $, then $ \Im\funcU(\beta) = 0 $ by \eqref{pf:cor:Vcomplex:d2:eq1} and since $\Re U>2|{\Im U_1}|^\frac12=0$, we obtain
with \eqref{pf:cor:Vcomplex:d2:ReIm}
\begin{align}\label{super}
	\Re\funcU(\beta) = \beta \Re U + \calO(\beta^2) > 0, \quad {\beta\to 0+}.
\end{align}
Hence, \eqref{pf:cor:Vcomplex:d2:cond(i)} holds and again
the claim follows from Theorem~\ref{theorem:s=2}(i).

\subsubsection{Proof of \upshape{(ii)} for \texorpdfstring{$d=2$}{d=2}}
Assume first that $\Im U \neq 0$. Then we conclude from \eqref{pf:cor:Vcomplex:d2:ReIm} that
\begin{align*}
    \frac{\Re\funcU(\beta)}{2|{\Im\funcU(\beta)}|^\frac12}
    =
    \frac{\beta\Re U + \calO(\beta^2)}{2\beta^\frac12|{\Im U - \beta \Im D}|^\frac12}
    =
    \calO(\beta^\frac12 )\rightarrow 0,
    \quad\beta\to 0+.
\end{align*}
Therefore, if $R>0$ is as in Theorem~\ref{theorem:s=2}(ii), we find
\begin{align} \label{pf:cor:Vcomplex:d2:cond(ii)}
    \Re\funcU(\beta) < 2|{\Im\funcU(\beta)}|^\frac12 \big( 1 - \const |{\Re\funcU(\beta)}| \big),
    \quad {\beta\to 0+},
\end{align}
so the claim follows from Theorem~\ref{theorem:s=2}(ii).

It remains to treat the case $\Re U < 2|{\Im U_1}|^\frac12 $.
By the
first part of this proof we can assume that $ \Im U = 0 $, so that again \eqref{pf:cor:Vcomplex:d2:eq1} holds.
If $ \Im U_1 \neq 0 $, then \eqref{pf:cor:Vcomplex:d2:ReIm} and \eqref{pf:cor:Vcomplex:d2:eq1} imply
\begin{align*}
    \frac{\Re\funcU(\beta)}{2|{\Im\funcU(\beta)}|^\frac12}
        =
    \frac{\Re U + \calO(\beta)}{2|{\Im U_1}|^\frac12}
        \rightarrow
    \frac{\Re U}{2|{\Im U_1}|^\frac12}
        < 1,
    \quad \beta \to 0+,
\end{align*}
and this leads to \eqref{pf:cor:Vcomplex:d2:cond(ii)}, and hence
the claim follows from Theorem~\ref{theorem:s=2}(ii).
If $ \Im U_1 = 0 $, then \eqref{pf:cor:Vcomplex:d2:eq1} implies $ \Im\funcU(\beta) = 0 $ and
from $\Re U < 2|{\Im U_1}|^\frac12=0$ we have $ \Re\funcU(\beta) < 0 $ as $ \beta\to 0+ $ by \eqref{pf:cor:Vcomplex:d2:ReIm};
cf.~\eqref{super}. 
This again implies \eqref{pf:cor:Vcomplex:d2:cond(ii)} and together with Theorem~\ref{theorem:s=2}(ii)
the claim follows.

\subsection{Details on Example~\ref{example:s=2.V.real.theta}} \label{subsec.example.proofs.1}

Here it is assumed that $ V $ is a non-negative potential, $ V \not\equiv 0 $, that satisfies \eqref{assumption:family_v}
and \eqref{def:Rollnik_space_mainjussi}, and the coupling is of the form $ \beta(\epsilon) = e^{i\theta} \epsilon $ with
$ \theta \in (-\pi, \pi] $ and $\epsilon>0$. Then one has
\begin{align} \label{ReImExample}
    \Re\beta(\epsilon) = \cos(\theta)\epsilon
    \quad\text{and}\quad
    \Im\beta(\epsilon) = \sin(\theta)\epsilon,
\end{align}
and $ U = \| V \|_{L^1} > 0 $. Hence, we may fix $ \constcor > 0 $
as in Corollary \ref{cor:s=2.V.real} and check the conditions \eqref{cond:(i).cor.V.real} and
\eqref{cond:(ii).cor.V.real} to prove \eqref{theta}.

For $ d = 1 $ it is clear from \eqref{ReImExample} that \eqref{cond:(i).cor.V.real} becomes
\begin{equation} \label{cos(theta)}
 \cos(\theta)>\left(-\frac{U_1}{\| V \|_{L^1}}+ R'|{\sin(\theta)}|\epsilon \right)\sin^2(\theta)\epsilon
\end{equation}
and \eqref{cond:(ii).cor.V.real} becomes
\begin{equation*}
 \cos(\theta)<\left(-\frac{U_1}{\| V \|_{L^1}}- R'|{\sin(\theta)}|\epsilon \right)\sin^2(\theta)\epsilon.
\end{equation*}
The non-negativity of $ V $ yields $ U_1 > 0 $, and hence $-U_1/\| V \|_{L^1}<0$. Therefore, for $ \epsilon\to 0+ $ 
we see that
$ \beta(\epsilon) $ satisfies \eqref{cond:(i).cor.V.real} if
$ \theta \in \left[-\pi/2, \pi/2\right] $ and $ \eqref{cond:(ii).cor.V.real} $ otherwise, i.e.,~\eqref{theta} 
in the case $d=1$ holds.
If $ d = 2 $, then it follows that \eqref{cond:(i).cor.V.real} is satisfied if $ \theta = 0 $, whereas $ \theta\neq 0 $
implies \eqref{cond:(ii).cor.V.real}, i.e.,~\eqref{theta} in the case $d=2$ holds.

It remains to prove that the spectral enclosure \eqref{ineq:Davies} is sharp
in the weak coupling regime.
In fact, for $ V $ and $ \beta(\epsilon) $ as above \eqref{ineq:Davies} on the one hand
implies
\begin{align} \label{spectral_inclusion_strict}
    \lambda \in \discBig{0}{\frac{\epsilon^2}{4} \| V \|_{L^1}^2},
\end{align}
and on the other hand, by squaring both sides of the expansion \eqref{expansion:example}, we see that the weakly coupled eigenvalue
$ \lambda_{\beta(\epsilon)} $ of $ H_{\beta(\epsilon)} $ satisfies
\begin{equation}
\begin{aligned} \label{expansion:example.squared}
    \lambda_{\beta(\epsilon)}
        &=
    -\frac{e^{2i\theta} \epsilon^2}{4} \| V \|_{L^1}^2 + \calO(\epsilon^3) ,
    \quad{\epsilon\to 0+}.
\end{aligned}
\end{equation}
By varying $ \theta \in [-\pi/2, \pi/2] $ it is now clear that for every point $ \lambda $ on the boundary of the disc \eqref{spectral_inclusion_strict} there exists
a Schrödinger operator, namely $ H_{\beta(\epsilon)} $, that has a weakly coupled eigenvalue $ \lambda_{\beta(\epsilon)} $ such that
\begin{align*}
    |\lambda-\lambda_{\beta(\epsilon)}| = \calO(\epsilon^3),
    \quad{\epsilon\to 0+},
\end{align*}
i.e., the sharpness of the enclosure
\eqref{spectral_inclusion_strict} in the weak coupling regime follows.

\subsection{Details on Example~\ref{example_d=2}}
\label{subsec.example.proofs.2}

Let $V$ be some potential that satisfies \eqref{assumption:family_v} and \eqref{def:Rollnik_space_mainjussi},
and, in addition, assume that
\begin{equation}\label{uvhier}
    \Im U = \int_{\R^2} \Im V(x) \dx = 0
    \quad\text{ and }\quad
    \Re U = \int_{\R^2} \Re V(x) \dx > 0.
\end{equation}
In this example for $ d = 2 $ we consider the family
\begin{align*}
    V(\alpha) := \alpha\Re V + i\Im V,\quad \alpha \in \R,
\end{align*}
and the corresponding Schrödinger operators $ H_\beta(\alpha)=-\Delta-\beta V(\alpha) $ in \eqref{H.beta.def}
for $ \beta\to 0+ $. Note that each $V(\alpha)$, $\alpha\in\R$, also satisfies the conditions \eqref{assumption:family_v} and \eqref{def:Rollnik_space_mainjussi}.

We claim that
\begin{align} \label{assertion_V_alpha_1}
    V(\alpha) \text{ satisfies }
    \begin{cases}
        \eqref{cond:(i).cor.V.complex} & \text{ if } \alpha > \alpha^*, \\
        \eqref{cond:(ii).cor.V.complex} & \text{ if } \alpha < \alpha^*, \: \alpha \neq 0,
    \end{cases}
\end{align}
where
\begin{align} \label{alpha*}
    \alpha^* := \frac{4}{\pi (\Re U)^2} \abs{\int_{\R^4} (\Re V(x)) \log|x-y| (\Im V(y)) \dmath(x,y)}.
\end{align}
In fact, to verify this define for $\alpha\in\R$
\begin{align*}
    U(\alpha) &:= \int_{\R^2} V(\alpha)(x) \dx, \\
    \quad
    U_1(\alpha) &:= \frac{1}{2\pi} \int_{\R^4} V(\alpha)(x) \log|x-y| V(\alpha)(y) \dmath(x,y),
\end{align*}
and notice that $ U(\alpha) $ and $ U_1(\alpha) $ are as in \eqref{def:U_U_1_intro}. Using \eqref{uvhier} it
is clear that
\begin{equation*}
 U(\alpha)=\int_{\R^2} \alpha (\Re V(x)) + i(\Im V(x)) \dx =
    \alpha \Re U,
\end{equation*}
and hence $ \Im U(\alpha) = 0 $ for $ \alpha \in \R $ and $ \Re U(\alpha) \neq 0 $ for $ \alpha \neq 0 $. Next, set
\begin{align} \label{f(alpha)}
    f(\alpha) := \Re U(\alpha) - 2|{\Im U_1(\alpha)}|^\frac12
        = \alpha\Re U - 2|{\Im U_1(\alpha)}|^\frac12, \quad \alpha\in \R.
\end{align}
Observe that $ f(\alpha) > 0 $ implies \eqref{cond:(i).cor.V.complex} and that
$ f(\alpha) < 0 $ implies \eqref{cond:(ii).cor.V.complex} for the potential $ V(\alpha) $.
A short computation yields
\begin{align*}
    \Im U_1(\alpha)
        &=
    \frac{1}{\pi} \int_{\R^4} (\Re V(\alpha)(x)) \log|x-y| (\Im V(\alpha)(y)) \dmath(x,y)  \\
        &=
    \frac{\alpha}{\pi} \int_{\R^4} (\Re V(x)) \log|x-y| (\Im V(y)) \dmath(x,y),
\end{align*}
and hence
\begin{align*}
    f(\alpha)
    =
    \alpha\Re U
    -
    2 \frac{|\alpha|^\frac12}{\sqrt{\pi}}  \abs{\int_{\R^4} (\Re V(x)) \log|x-y| (\Im V(y)) \dmath(x,y)}^\frac12.
\end{align*}
It is now easily seen that the (at most) two zeros of $ f $ are given by $ \alpha = 0 $ and
\eqref{alpha*}
as well as that $ f(\alpha) < 0 $ if $ \alpha < \alpha^* $, $ \alpha \neq 0 $, and $ f(\alpha) > 0 $
if $ \alpha > \alpha^* $. By the definition \eqref{f(alpha)} of $ f $ this implies \eqref{assertion_V_alpha_1}.

\appendix

\section{} \label{section:appenidx}

\subsection{Inequalities related to the Green's function} \label{section:appenidx:inequalities}
Recall that for $ z \in \C_+ $  the integral kernel of the resolvent $ (-\Delta+z^2)^{-1} $
of the free Laplacian in $ L^2(\R^d) $ is given by
\begin{equation} \label{def:resolvent_appendix}
\begin{aligned}
    \ResolventKernel{d}(x,y;z)
        =
    \begin{cases}
        (2z)^{-1} e^{-z\abs[]{x-y}} & \text{ if } d = 1, \\
        (2\pi)^{-1} K_0(z\abs[]{x-y}) & \text{ if } d = 2,
    \end{cases}
    \qquad x,y \in \R^d,
\end{aligned}
\end{equation}
where $ K_0 : \C \setminus (-\infty,0] \rightarrow \C $
denotes the modified Bessel function of the second kind of order zero.

For the convenience of the reader, let us briefly recall some properties of $ K_0 $ (for an extensive treatment, we refer to 
\cite[Chap.~9.6]{AbramowitzStegun}).
First, $ K_0 : \C \setminus (-\infty, 0] \rightarrow \C $
is a holomorphic function, that by \cite[Eq.~(9.6.13)]{AbramowitzStegun} admits the expansion
\begin{align}    \label{powerseries:K0}
    K_0(w) = -\log(w) + \log(2) - \gamma + \calO( w^2 \log(w) ),
    \quad w \to 0, \: w \in \C_+,
\end{align}
where $ \log : \C \setminus (-\infty, 0] \rightarrow \C $ is the principal branch of the complex logarithm, and
$ \gamma $ is the Euler-Masceroni constant. In particular, $ K_0 $ has a logarithmic singularity at $ w = 0 $.
For large $ w $ with non-negative real part, it follows from \cite[Eq.~(9.7.2)]{AbramowitzStegun} that $ K_0 $ has
the asymptotic behavior
\begin{align}    \label{property:K0_decays_exponentially}
    K_0(w) = \left( \frac{\pi}{2w} \right)^{\frac12} e^{-w} \left( 1 + \calO(w^{-1}) \right),
    \quad w \to \infty, \: \Re w \geq 0,
\end{align}
so $ K_0(w) $ tends to zero uniformly if $ w \rightarrow \infty $
with $ \Re w \geq 0 $.

Since we have to deal with the complex logarithm, let us briefly summarize
some properties of its principal branch. First, there holds
\begin{align}    \label{ineq:log(z)+pi}
    \logabs{w} \leq |{\log(w)}| < \logabs{w} + \pi,
    \quad w \in \C \setminus (-\infty, 0],
\end{align}
and as a consequence (using $\pi < C |{\log(1/2)}|< C|{\log |w|}|$ for $|w|<1/2$)
\begin{align}    \label{ineq:log(z)}
    \logabs{w} \leq |{\log(w)}| \leq C\logabs{w},
    \quad w \in \discBig{0}{\frac12} \setminus (-\infty, 0].
\end{align}
Furthermore, it is not difficult to verify the identity
\begin{equation}     \label{identitiy:log(z)}
\begin{aligned}
    \log(w_1 w_2) = \log(w_1) + \log(w_2) \spacer{if}{\quad} |{\arg(w_1) + \arg(w_2)}| < \pi
\end{aligned}
\end{equation}
for all $ w_1, w_2 \in \C \setminus (-\infty, 0] $,
where $ \arg(w) \in (-\pi, \pi) $ denotes the argument of a complex number
$ w \in \C \setminus (-\infty, 0] $.

In the following, recall that for $ z \in \C_+ $ we have
\begin{align} \label{def:funcg}
    \funcg{d}(z)
        =
    \begin{cases}
        (2z)^{-1} & \text{if } d = 1, \\
        -(2\pi)^{-1} \log (z) & \text{if } d = 2,
    \end{cases}
\end{align}
and, in addition, define for $ x, y \in \R^d $ (with $ x \neq y $ if $ d = 2 $),
\begin{align} \label{def:funch}
    \funch{d}(x,y;z)
        =
    \begin{cases}
        (2z)^{-1} \left( 1 - z\abs[]{x-y} \right) & \text{ if } d = 1, \\
        -(2\pi)^{-1} \left( \log(z|x-y|) - \log(2) + \gamma \right) & \text{ if } d = 2.
    \end{cases}
\end{align}

In the next lemma, we collect useful inequalities for certain integral kernels that appear in the analysis
of the Schrödinger operator $ H_\beta $.

\begin{lemma}  \label{lemma:ineq:weber_kernels}
For $ z \in \C_+ $ and $ x, y \in \R^d $, let $ \ResolventKernel{d}(x,y;z) $ be as in \eqref{def:resolvent_appendix}. Let $ \funcg{d}(z) $ and $ \funch{d}(x,y;z) $ be given by \eqref{def:funcg} and \eqref{def:funch}, respectively. Then, the following inequalities hold.
\begin{enumerate}
    \item If $ d = 1 $, then
    \begin{enumerate}[\upshape (a)]
        \item $ |\ResolventKernel{1}(x,y;z) - \funcg{1}(z)| \leq \frac{1}{2}\abs{x-y} $,
        \item $ |\ResolventKernel{1}(x,y;z) - \funch{1}(x,y;z)| \leq \frac{1}{4}\abs[]{z}\abs[]{x-y}^{2} $.
    \end{enumerate}
    \item If $ d = 2 $, then for $ |z| < 1/2 $ there exists a constant
    $ C > 0 $, such that
    \vspace{-0.05cm}
    \begin{enumerate}[\upshape(a)]
        \item $ |\ResolventKernel{2}(x,y;z) - \funcg{2}(z)|^2 \leq C(1 + \logabs{x-y}^2) $,
        \item $ |\ResolventKernel{2}(x,y;z) - \funch{2}(x,y;z)| \leq C(1 + \logabs{x-y}) $,
        \item $ |\ResolventKernel{2}(x,y;z) - \funch{2}(x,y;z)| \leq C |z||{\log(z)}|
        $ if $ |x-y| < |z|^{-\frac12} $.
    \end{enumerate}
\end{enumerate}
\end{lemma}

\begin{proof}
(i)
For $ d = 1 $, there holds
\begin{align*}
    &\ResolventKernel{1}(x,y;z) - \funcg{1}(z) = \frac{e^{-z|x-y|}-1}{2z}, \\
    &\ResolventKernel{1}(x,y;z) - \funch{1}(x,y;z) = \frac{e^{-z|x-y|}-1+z|x-y|}{2z}.
\end{align*}
In order to see (a), notice that for any $ z \in \C_+ $ one has
\begin{align*}
    \abs[]{\frac{e^{-z|x-y|}-1}{2z}}
        =
    \abs[]{\frac{1}{2z} \int_{0}^{\abs[]{x-y}} \frac{\dmath}{\dt} e^{-zt} \dt}
        =
    \frac{1}{2} \abs[]{\int_{0}^{\abs[]{x-y}} e^{-zt} \dt}
        \leq
    \frac{1}{2}|x-y|,
\end{align*}
i.e., the claimed inequality.
Estimate (b) follows from integration by parts,
\begin{align*}
        \frac{e^{-z\abs[]{x-y}} - 1 + z \abs[]{x-y}}{2z}
    &=
        \frac{1}{2z} \int_{0}^{\abs[]{x-y}} (\abs[]{x-y}-t) \frac{\dmath^2}{\dt^2} e^{-zt} \dt \\
    &=
        \frac{z}{2} \int_{0}^{\abs[]{x-y}} (\abs[]{x-y}-t)  e^{-zt} \dt,
\end{align*}
and hence
\begin{align*}
    \abs[]{\frac{e^{-z\abs[]{x-y}} - 1 + z \abs[]{x-y}}{2z}}
        \leq
    \frac{|z|}{2} \int_{0}^{\abs[]{x-y}} (\abs[]{x-y} - t) \dt
        =\frac{1}{4}
    |z||x-y|^2.
\end{align*}
\noindent
(ii) Consider $ x, y \in \R^2 $ with $ x \neq y $  and set
\begin{align} \label{proof:def:w}
    w := z|x-y|
\end{align}
for $ z \in \C_+ $ with $ |z| < 1/2 $; it follows that $ w \neq 0 $ and $ w \in \C_+ $.
With $ w $ in \eqref{proof:def:w}, the kernels we have to estimate read as
\begin{align}
    &\ResolventKernel{2}(x,y;z) - \funcg{2}(z) = \frac{K_0(w)+\log (z)}{2\pi} \label{proof:weber_2D_1}, \\
    &\ResolventKernel{2}(x,y;z) - \funch{2}(x,y;z) = \frac{K_0(w)+\log(w)-\log(2)+\gamma}{2\pi}. \label{proof:weber_2D_2}
\end{align}

For the inequalities below recall that we use the convention that 
$ C $ denotes a positive constant that may change in
between the estimates.
Let us start by proving the claim (a).
If $ |w| < 1 $, then \eqref{powerseries:K0} implies
\begin{align*}
    | K_0(w) + \log(w) |^2
        \leq
    C \left( 1 + |w|^2 |{\log(w)}| \right)^2
        \leq
    C.
\end{align*}
The last inequality together with (see \eqref{identitiy:log(z)})
\begin{align} \label{proof:logarithmic_identity}
    \log (z) = \log(w) - \log |x-y|
\end{align}
and the elementary inequality $ (|a|+|b|)^2 \leq 2|a|^2 + 2|b|^2 $ yield 
\begin{align*}
    |K_0(w) + \log (z)|^2
        =
    |K_0(w) + \log(w) - \log|x-y||^2
        \leq
    C(1 + \logabs{x-y}^2).
\end{align*}
If $ |w| \geq 1 $, then \eqref{property:K0_decays_exponentially} and \eqref{ineq:log(z)} yield (recall that $|z|<1/2$)
\begin{equation*}
\begin{split}
        \abs[]{K_0(w) + \log (z)}
    \leq
        |K_0(w)| + |{\log (z)}|
    \leq
        C + |{\log (z)}|
        \leq
    C (1+\log{(|z|^{-1})}).
\end{split}
    \end{equation*}
Since $ |z| < 1/2$ and $ 2< |z|^{-1} \leq |x-y| $ if $ |w| \geq 1 $ by \eqref{proof:def:w}, we arrive at
\begin{equation*}
    \abs[]{K_0(w) + \log (z)}^2
    \leq
    C (1+|{\log|x-y|}|^2).
\end{equation*}
In summary, the inequality (a) holds.

It is not difficult to see that (b) follows from (a);
indeed, taking the square root in (a) and applying
the elementary inequality $ (|a|^2+|b|^2)^\frac12 \leq |a| + |b| $ gives us
\begin{align*}
    \abs{K_0(w) + \log (z)} \leq C( 1 + |{\log|x-y|}|),
\end{align*}
so with \eqref{proof:logarithmic_identity} and the triangle inequality
\begin{align*}
     |{K_0(w)+\log(w)-\log(2)+\gamma}|
         &\leq
     |{K_0(w)+\log (z)}| + |{\log|x-y|}| + C \\
         &\leq
     C(1 + |{\log|x-y|}|) + |{\log|x-y|}| + C \\
         &=
     C(1 + |{\log|x-y|}|),
\end{align*}
which is the claim of (b).

Finally, to show (c), suppose that $|x-y| \leq |z|^{-\frac12}$, which is equivalent to
\begin{align} \label{proof:t}
 |w| \leq |z|^{\frac12}.
\end{align}
By rearranging \eqref{powerseries:K0} and using \eqref{identitiy:log(z)} (we have $ |{\arg{w}}| < \pi/2 $) one verifies
\begin{align*}
    |{K_0(w) + \log(w) - \log(2) + \gamma}|
        \leq
    C|{w^2 \log(w)}|
        =
    C\abs{\frac{w^2}{2} \log(w^2)}
            \leq
    C|w|^2 |{\log(w^2)}|
\end{align*}
and hence with \eqref{ineq:log(z)} (note that $ |w|^2 \leq |z| < 1/2 $ by \eqref{proof:t})
\begin{align*}
    |K_0(w) + \log(w) - \log(2) + \gamma|
        \leq
    C |w|^2 |{\log|w|^2}|.
\end{align*}
Finally, a combination of \eqref{proof:t} and the fact that $ (0,1) \ni x \mapsto x^2 \left( |{\log(x^2)}| + 1 \right) $ is monotonously increasing yields
\begin{align*}
    |K_0(w) + \log(w) - \log(2) + \gamma|
        &\leq
        C |w|^2 \left( |{\log|w|^2}| + 1 \right) \\
            &\leq
        C |z| \left( |{\log|z|}| + 1 \right) \\
            &\leq
        C |z| |{\log|z|}|  \\
            &\leq
        C|z||{\log(z)}|,
\end{align*}
where we again used \eqref{ineq:log(z)} for the last inequality;
(c) is shown. \qedhere
\end{proof}

\subsection{Eigenvalues of operator functions} \label{section:eigenvalues_operator_functions}

This section follows the presentation of \cite[Chap.~XI.9]{Gohberg}.
We assume that $ \calH $ and $ \calG $ are Hilbert spaces and that $ \Omega \subset \C $ is a non-empty, open, and connected set.

For an analytic operator function $ W : \Omega \rightarrow \mathcal{L}(\calH) $
we call $ z_0 \in \Omega $ an $ \textit{eigenvalue} $ of $ W $ if
$
    \ker{W(z_0)} \neq \lbrace 0 \rbrace.
$
For an eigenvalue $ z_0 \in \Omega $ the quantity
$ \dim\ker W(z_0) $ is called the \textit{geometric multiplicity} of $ W $ at $ z_0 $.
If $ z_0 \in \Omega $ is an eigenvalue of $W$ and, in addition,
 $ W(z_0) $ is Fredholm (that is, $ \ran{ W(z_0) } $ is closed and both
    $ \dim \ker{W(z_0)} $ and $ \codim\ran W(z_0) $
    are finite) and $ W(z) $ is boundedly invertible for all
    $ 0 < \abs[]{z-z_0} < \epsilon $ for some $ \epsilon > 0 $, then
 $ z_0 $ is said to be an \textit{eigenvalue of finite type} of $ W $.
If $ z_0 $ is an eigenvalue of finite type of $ W $, then by \cite[Thm.~XI.8.1]{Gohberg} $ W(z) $ admits a factorization
\begin{align} \label{operator_function_equivalence}
    W(z) = E(z) \left( P_0 + (z-z_0)^{\kappa_1} P_1 + \ldots + (z-z_0)^{\kappa_r} P_r \right) F(z)
\end{align}
in a small neighborhood $ D $ of $ z_0 $, where $ E : D \rightarrow \calL(\calH) $ and $ F : D \rightarrow \calL(\calH) $ are analytic operator functions,
    boundedly invertible on $ D $,
    $$ 1 \leq \kappa_1 \leq \kappa_2 \leq \ldots \leq \kappa_r,$$
    are positive integers, $ P_1, \ldots, P_r $ are mutually disjoint projections of rank one, and
    $$ P_0 = I_{\calH} - P_1 - \ldots - P_r;$$
here $ I_\calH$ denotes the identity operator in  $ \calH $.
The sum $ \kappa_1 + \ldots + \kappa_r $
is called the \textit{algebraic multiplicity} of
$ W $ at $ z_0 $, and will be denoted by
$ m_a(z_0; W) $.

\begin{remark} \label{remark:algebraic_multiplicity_root}
Observe that in the scalar case, that is, $W : \Omega \rightarrow \C $ is a complex-valued
holomorphic function, the notions above simplify accordingly. More precisely, eigenvalues are zeros of $W$ and each eigenvalue
is an eigenvalue of finite type, its geometric multiplicity is one, and its algebraic multiplicity
coincides with the
multiplicity of the zero of $W$.
\end{remark}

It follows from the factorization \eqref{operator_function_equivalence} that the eigenvalues of finite type  of an operator function
and their respective algebraic multiplicities are invariant under multiplications by boundedly invertible operator functions.

\begin{lemma} \label{lemma:invertible_operator_eigenvalues}
Let $ W : \Omega \rightarrow \calL(\calH) $ and $ R : \Omega \rightarrow \calL(\calH) $
be analytic operator functions, and suppose that $ R(z) $ is boundedly invertible for all
$ z \in \Omega $.
Then $ z_0 \in \Omega $ is an eigenvalue of finite type of $ W $ if and only if
it is an eigenvalue of finite type of $ R W $ and $ W R $, in which case
\begin{align*}
    m_a(z_0; W)
        =
    m_a(z_0; RW)
    		=
    	m_a(z_0; WR).
\end{align*}
\end{lemma}

In the following lemma $ I_\calH$ and $I_\calG$ denote the identity operators in the Hilbert spaces $ \calH $
and $\calG$, respectively. It turns out that operator functions of the special form
$ I_\calG + AB $ and $ I_\calH + BA $ have the same eigenvalues of finite type. 
For the proof we make use of Fredholm determinants; for a definition and basic properties of the
latter we refer to \cite{SimonTraceIdeals} and 
for further details to \cite[App.~C]{GesztesyTraces}.

\begin{lemma} \label{lemma:algebraic_multiplicity_A_B}
Let $ A : \Omega \rightarrow \calL(\calH,\calG) $ and $ B : \Omega \rightarrow \calL(\calG, \calH) $ be
analytic families of Hilbert-Schmidt operators. 
Then $ z_0\in\Omega $ is an eigenvalue of finite of $ I_\calG + AB $ if and only
if it is an eigenvalue of finite type of $ I_\calH + BA $, in which case
\begin{align*}
    m_a(z_0; I_\calG + AB) = m_a(z_0; I_\calH + BA).
\end{align*}
\end{lemma}

\begin{proof}
Recall that for an analytic family $ W : \Omega \rightarrow \calL(\calH) $ of trace class operators
the Fredholm determinant
\begin{align} \label{det}
	\Omega\ni z \mapsto \det(I_\calH + W(z))
\end{align}
is an analytic function. Moreover, the zeros of \eqref{det} are
exactly the eigenvalues of finite type of $ I_\calH + W $ and their respective algebraic
multiplicities coincide with their multiplicity as a root of $ {\det(I_\calH+W)} $. Hence,
the claim follows from the well-known identity (see, e.g., \cite[Thm.~C.4.4(iv)]{GesztesyTraces}
\begin{align*}
	\det(I_\calG + A(z)B(z)) = \det(I_\calH + B(z)A(z)),
	\quad z \in \Omega,
\end{align*}
which is satisfied since by assumption $ A(z)B(z) $ and $ B(z)A(z) $ are trace class for every $ z \in \Omega $;
cf.~\cite[Satz~3.23]{WeidmannLineareOperatorenTeil1}.
\end{proof}

\subsection{Useful tools from complex analysis}

We collect useful results from complex analysis. The first lemma concerns the
 multiplicity of roots under substitutions.

\begin{lemma} \label{lemma:multiplicity}
Let $ \Omega_1, \Omega_2 \subset \C $ be open, let $ h : \Omega_1 \rightarrow \C $ be holomorphic, let
$ {\Psi : \Omega_2 \rightarrow \Omega_1} $ be holomorphic and bijective, and assume that
$ h(\Psi(w_0)) = 0 $ for some $ w_0 \in \Omega_2 $.
Then the multiplicity of $ \Psi(w_0) $ as a root of the function $ h $ coincides with the multiplicity of $ w_0 $ as a root of the function $ h \circ \Psi $.
\end{lemma}

\begin{proof}
Let $ \Gamma = \lbrace w \in \C : |w-w_0| = r \rbrace $, where $ r > 0 $ is chosen such that
$ \Gamma $ and its interior lie in $ \Omega_2 $, and suppose that $ \gamma $ is a simple, positively oriented parametrization of $ \Gamma $.
The argument principle for holomorphic functions then implies
\begin{align*}
    m_a(w_0; h\circ\Psi)
        &=
    \frac{1}{2\pi i} \int_{\gamma} \frac{(h \circ \Psi)'(w)}{(h \circ \Psi)(w)} \dmath w
        =
    \frac{1}{2\pi i} \int_{\gamma} \frac{h'(\Phi(w))\Psi'(w)}{h(\Psi(w))} \dmath w \\
        &=
    \frac{1}{2\pi i} \int_{\Psi(\gamma)} \frac{h'(z)}{h(z)} \dmath z
        =
    m_a(\Psi(w_0); h),
\end{align*}
where in the last equality we used that $ \Psi(\gamma) $ is a simple, smooth and positively oriented curve around $ \Psi(w_0) $;
this, in turn, holds as
$ \Psi $ is assumed to be holomorphic and bijective.
\end{proof}

The following lemma, based on Rouche's theorem, is an important ingredient
to find zeros of the  function $\fbeta{d} \circ \Phi $ in Proposition~\ref{prop:equivalence}.

\begin{lemma} \label{lemma:rouche_eigenvalue_abstract}
Let $ a \in \C \setminus \lbrace 0 \rbrace $, let $\calO_a$ be an open neighborhood of $a$ such that
$0\not\in\calO_a$, let $ b:\calO_a\rightarrow\C $, and consider the function
\begin{align} \label{f_eps_rouche}
    h(w) = 1 - \frac{1}{w} (a + b(w)),\quad w\in\calO_a.
\end{align}
If for some $\delta \in(0, \tfrac 1{12})$ the following conditions hold,
\begin{enumerate}
    \item $ \disc{a}{4|a|\delta} \subset \calO_a$,
    \item $ b $ is holomorphic on $ \disc{a}{4|a|\delta} $,
    \item $ |b(w)| < \delta |a| $ for all $ w \in \disc{a}{4|a|\delta} $,
\end{enumerate}
then the function $ h $ has exactly one root in the disc $ \disc{a}{|a|\delta} $ and this root is simple.
\end{lemma}

\begin{proof}
We start by proving the existence of a root, which is essentially a consequence of Rouche's theorem.
Consider the mapping
\begin{align*}
    \Psi: \disc{0}{3\delta} \rightarrow \C,
    \quad
    \alpha \mapsto \Psi(\alpha)=\frac{a}{1-\alpha};
\end{align*}
since $|\alpha|< 1/4$, $\Psi$ is well-defined, injective, and holomorphic.
Moreover, $|\alpha|< 1/4$ also implies
\begin{align*}
    |a - \Psi(\alpha)|
        =
    \abs{a - \frac{a}{1-\alpha}}
        \leq
    |a| \frac{|\alpha|}{1 - |\alpha|}
        <
	|a| \frac{3 \delta}{1 - \frac{1}{4}}
        =
    4 |a| \delta,
\end{align*}
which shows that $ \Psi $ maps $ \disc{0}{3\delta} $ into $ \disc{a}{4|a|\delta} $.
In particular, the function $f=h \circ \Psi $ is holomorphic on $ \disc{0}{3\delta} $ due to assumption (i) and (ii).
A direct computation shows
\begin{align*}
    f(\alpha)=h(\Psi(\alpha))
        =
     1 - \frac{1}{\Psi(\alpha)} (a + b(\Psi(\alpha)))
        =
    \alpha - (1 - \alpha) \frac{b(\Psi(\alpha))}{a},\quad\alpha\in\disc{0}{3\delta} ,
\end{align*}
and setting $ g(\alpha) = -\alpha $ we obtain for $ \alpha \in \partial\disc{0}{2\delta} $ the estimate
\begin{align*}
    |f(\alpha)+g(\alpha)|
        \leq
    |1 - \alpha| \frac{|b(\Psi(\alpha))|}{|a|}
        <
    (1 + 2\delta)\delta
        <
    2\delta
        =
    |\alpha|\leq|f(\alpha)|+|g(\alpha)|,
\end{align*}
where we used assumption (iii) for the second inequality and $2\delta<1/6$ for the third inequality.
Therefore, we can apply
Rouche's theorem (see, e.g., \cite[Chap.~V, Thm.~3.8]{Conway}) to the functions $ f =h \circ \Psi$
and $ g $ on the disc
$ \disc{0}{2 \delta} $, and conclude that $ f= h \circ \Psi $ has exactly one zero $ \alpha_\delta $ in $ \disc{0}{2\delta} $,
which, in addition, is simple. Therefore, it follows that $ w_\delta := \Psi(\alpha_\delta) $ is a zero
of $ h $ in $ \disc{a}{4|a|\delta} $. Since the function
$ \Psi : \disc{0}{3\delta} \rightarrow \disc{a}{4|a|\delta} $ is holomorphic and injective,
thus bijective onto its range, $ w_\delta $ is a simple zero of $ h $ by Lemma~\ref{lemma:multiplicity}.

It remains to prove that $ w_\delta $ belongs to the smaller disc $ \disc{a}{|a|\delta} $ and that it
is the only zero of $ h $ inside this disc.
To see the first claim, note that $ h(w_\delta) = 0 $ and assumption (iii) imply
\begin{align*}
    |a - w_\delta|
        =
    |b(w_\delta)|
        <
    \delta |a|.
\end{align*}
To show that $ w_\delta $ is the only zero of $h$ in $ \disc{a}{|a|\delta} $ consider
$ \widetilde{w} \in \disc{a}{|a|\delta} $
such that $ h(\widetilde  w)=0 $ and set
\begin{align*}
    \widetilde{\alpha} := \frac{\widetilde{w} - a}{\widetilde{w}}.
\end{align*}
Then we have $ \Psi(\widetilde{\alpha}) = \widetilde{w} $ and
\begin{equation*}
|\widetilde{\alpha}|=\frac{|\widetilde{w} - a|}{|\widetilde{w}|}\leq \frac{|a|\delta}{|a|-|a|\delta}
=\frac{\delta}{1-\delta}<2\delta,
\end{equation*}
that is, $\tilde{\alpha} \in \disc{0}{2\delta}$.
Moreover, since $ \widetilde{w} $ is a root of $ h $, it follows that $ \widetilde{\alpha} $ is a root of $ h \circ \Psi $.
However, since $ \alpha_\delta $ is the only zero of $ h \circ \Psi $ in $ \disc{0}{2\delta} $, this implies
$ \widetilde{\alpha} = \alpha_\delta $ and thus $ \widetilde{w} = w_{\delta} $.
\end{proof}

\bibliography{../literatur}
\bibliographystyle{acm}

\end{document}